\documentclass[10pt,a4paper]{article}
\usepackage{amsmath}
\usepackage{amsfonts}
\usepackage{amssymb}
\usepackage{amsthm}
\usepackage{amscd} 
\usepackage[dvips]{graphicx}
\def\ui{{\mathbf i}} \def\ZZ{{\mathbb Z}} 
\def\f#1#2{\frac{#1}{#2}}  \def\exp#1{e^{#1}} 
\def\cb#1#2{  {\footnotesize \left( \!\!\bi{c} #1 \\ #2 \ei \!\!\right) } }
\def\ba#1{\left[\begin{array}{#1}}  \def\ea{\end{array}\right]}
 \def\fa{{\forall}}
\def\al{{\alpha}}  \def\de{{\delta}}

 \def\om{{\omega}}

\def\um{\frac{1}{2}}
\def\bi#1{\begin{array}{#1}}  \def\ei{\end{array}}
\def\beq{\begin{equation}} \def\eeq{\end{equation}}
\def\o{\overline}
   
\def\f#1#2{\frac{#1}{#2}}

\def\R#1{\sqrt{#1}} 

\def\addots{\mathinner{\mkern1mu\raise1pt\vbox{\kern7pt\hbox{.}}%
\mkern2mu
    \raise4pt\hbox{.}\mkern2mu\raise7pt\hbox{.}\mkern1mu}}

 \def\and{\hbox{and}}  
  
\def\diag{\hbox{\,diag\,}}  
  
 \def\mod{\hbox{\,mod\,}} 

\def\v{{\mathbf v}}  \def\x{{\mathbf x}}
\def\z{{\mathbf z}} \def\b{{\mathbf b}} \def\e{{\mathbf e}}
\def\a{{\mathbf a}}  \def\q{{\mathbf q}}
 \def\vf{{\mathbf f}} 
\def\w{{\mathbf w}} \def\u{{\mathbf u}}  
\def\c{{\mathbf c}} \def\d{{\mathbf d}}

\def\vn{{\mathbf 0}}  \def\ui{{\mathbf i}}

 \def\L{{\mathcal L}} 
  
  \def\I{{\mathcal I}}

  \def\NN{{\mathbb N}}
\def\ZZ{{\mathbb Z}}   
\def\CC{{\mathbb C}} \def\CCn{{\CC^n}} \def\CCnn{{\CC^{n\times n}}}
\def\RR{{\mathbb R}}   
\usepackage[width=16.6cm, left=2.2cm]{geometry}
\newtheoremstyle{teorema}{8pt}{8pt}{}{}{\bfseries}{}{8pt}{}
\theoremstyle{teorema}
\newtheorem{theorem}{Theorem}[section]
\newtheorem{proposition}[theorem]{Proposition}

\newtheorem{lemma}[theorem]{Lemma}

\newtheoremstyle{definizione}{8pt}{8pt}{\itshape}{}{\scshape}{}{8pt}{}
\theoremstyle{definizione}

\usepackage[usenames,dvipsnames]{color}

\title{Bernoulli, Ramanujan, Toeplitz and the triangular matrices}
\author{Carmine Di Fiore, Francesco Tudisco, Paolo Zellini\\
Dipartimento di Matematica, Universit\`a di Roma ``Tor Vergata''}
\begin{document}

\maketitle

\begin{abstract}
By using one of the definitions of the Bernoulli numbers,
we prove that they solve particular {\it odd} and {\it even} lower
triangular Toeplitz (l.t.T.) systems of equations.
In a paper Ramanujan writes down a sparse lower triangular
system solved by Bernoulli numbers; we observe that such system is
equivalent to a sparse l.t.T. system.
The attempt to obtain the sparse l.t.T. Ramanujan system
from the l.t.T. odd and even systems, has led us to study
efficient methods for solving generic l.t.T. systems. Such methods are
here explained in detail in case $n$, the number of equations, is a power
of $b$, $b=2,3$ and $b$ generic.
\end{abstract}



\indent \quad{\it Corresponding author (email address):} difiore@mat.uniroma2.it

\section*{Introduction} 
\label{sec:introduction}

The $j$-th Bernoulli number, $B_{2j}(0)$, is a rational number defined for any
$j\in\NN$, positive if $j$ is odd and negative if $j$ is even, whose denominator is
known, in the sense that it is the product of all prime numbers $p$ such that $p-1$
divides $2j$ \cite{Caley}, and, instead, only partial information are known about the numerator
\cite{Bn_Numer_Wagstaff}, \cite{Bn_Numer_Kellner}, \cite{Stein}.
Shortly, $B_{2j}(0)$, $j\geq 1$, could be defined by the well known Euler formula
$B_{2j}(0)=(-1)^{j+1}\f{2(2j)!}{(2\pi)^{2j}}\sum_{k=1}^{+\infty} \f{1}{k^{2j}}$, involving
the Zeta-Riemann function \cite{Bn_EulForm_Apostol}, \cite{BookOnBernoulli}.
May be the latter formula alone is sufficient to justify the
past and present interest in investigating Bernoulli numbers (B.n.).
Note that an immediate consequence of the Euler formula is the fact that the
$B_{2j}(0)$ go to infinite as $j$ diverges.

In literature one finds several identities involving B.n., and also several
``explicit'' formulas for them, which may appear more explicit than Euler formula
since involve finite (instead of infinite) sums \cite{Bn_ExplForm_Gould}, \cite{Bn_Mazur19onArsConjANDapplicinPureMath}, \cite{Bn_Comput_Kaneko}, \cite{Bn_ExplForm_Bucur}, \cite{Bn_ExplForm_Vella}, \cite{Bn_ExplForm_Ahmed}, \cite{Bn_Comput_Brent}, \cite{Bn_HigherOrderRec_Takashi}. Some of such identities/formulas
have been used to define algorithms for the computation of the numerators of
the B.n.. It is however interesting to note that there are efficient algorithms for
such computations which exploit directly the expression of the B.n. in terms of the Zeta-Riemann function \cite{Bn_Comput_CalcbnZR_BernoulliOrg}, \cite{Bn_Comput_McGown}, \cite{Stein}, \cite{Harvey}. See also \cite{Bn_Comput_AkiTani}, \cite{Bn_Comput_Kaneko}, \cite{Bn_Comput_Riordan_Merlini}.
As it is noted in \cite{Bn_Mazur19onArsConjANDapplicinPureMath}, the B.n. appear in several fields of mathematics; in particular, the numerators of the B.n. and their factors play an important role in advanced number theory (see \cite{Bn_MazurCyclotomicFields}, \cite{Bn_MazurCyclotomicFieldsCorrections}, \cite{Bn_Numer_Wagstaff}, \cite{MazurSteinBook}, \cite{Bn_Numer_Kellner}).
So, wider and wider lists of the ``first'' B.n. have been and are compiled, and also lists of the known factors of their numerators. The updating of these lists requires the implementation of efficient primality-test/integer-factorization algorithms on
powerful parallel computers. For instance, by this way the numerator of $B_{200}(0)$ first has been proved not prime, and then has been factorized as the product of five
prime integers. Two of such factors, respectively of 90 and 115 digits, have been found only very recently \cite{B200}, \cite{FactBigIntegersTogether}.

A lower triangular Toeplitz (l.t.T.) matrix $A$ is a matrix such that $a_{ij}=0$ if
$i<j$, and $a_{i,j}=a_{i+1,j+1}$, for all $i,j$. The product of two l.t.T. matrices
whatever order is used generates the same matrix, and such matrix is l.t.T..
Non singular l.t.T. matrices have an inverse which is l.t.T., and thus is
uniquely defined by its first column. Such remarks simply
follow from the fact that the set of all l.t.T. matrices is nothing else than
the set $\{p(Z)\}$ of all polynomials in the {\it lower-shift} matrix $Z=(\de_{i,j+1})$, and the fact that $\{p(X)\}$ is, for any choice of $X$, a commutative matrix algebra closed under inversion.

Note that, given a $n\times n$ l.t.T. matrix $A$, multiplying $A$ by a vector (${\mathcal M}$), or solving a system whose coefficient matrix is $A$ (${\mathcal S}$), are both operations that can be performed in at most $O(n\log n)$ arithmetic operations, thus in an amount of operations
significantly smaller than, for example, the $n(n+1)/2$ multiplications required by
the standard algorithms for lower triangular (non Toeplitz) matrices.
Such performances are possible by introducing alternative algorithms
which exploit, first, the strict relationship between the Toeplitz structure and the discrete Fourier transform \cite{DFTcirc}, and second, the fast implementation, known as FFT, of the latter. However, for (${\mathcal M}$) and (${\mathcal S}$) it is not so clear what is the best possible alternative algorithm. In particular, the algorithms performing the multiplication l.t.T. matrix $\times$ vector hold unchanged if the l.t.T. is replaced by a generic (full) Toeplitz matrix; so one guesses that better algorithms may be introduced, ad hoc for the l.t.T. case.
Analogously, a widely known exact algorithm able to solve l.t.T. systems (or, more precisely, to compute the first column of the inverse of a l.t.T. matrix) in at most
$O(n\log n)$ a.o., has essentially a recursive character, which is not so convenient from the point of view of the space complexity \cite{ltT_exactsolver_Morf}.
In order to avoid such drawback, however, one could use approximation inverse algorithm \cite{ltT_approxsolver_Bini}, \cite{ltT_approxsolver_Ng}. See also \cite{ltT_TrenchInverse}, \cite{ltT_Trench2NearlyTriangularsyst}, \cite{ltT_TrenchInversBandToep}, \cite{ltT_BoundsInverses}, \cite{ltT_Bound1norm}, \cite{ltT_QuestionOnInverse}, and the references in \cite{ltT_approxsolver_Ng}.

In this paper we emphasize the connection (may be also noted elsewhere,
see f.i. \cite{CollegaFasino}) between
Bernoulli numbers and lower triangular Toeplitz matrices. This connection will
finally result into new possible algorithms for computing simultaneously the first $n$
Bernoulli numbers.
More precisely, we prove that the vector
$\z=(\, B_{2j}(0)x^{j}/(2j)! \,)_{j=0}^{+\infty}$, $x\in\RR$ ($B_0(0)=1$),
solves three {\it type I}\, l.t.T. semi-infinite linear systems $A\x=\vf$, named {\it even}, {\it odd} and {\it Ramanujan}, respectively.
To such systems correspond other three systems, of {\it type II},
solved by the vector $Z^T\z=(\, B_{2j}(0)x^{j}/(2j)! \,)_{j=1}^{+\infty}$.
Type I ad II l.t.T. systems have been obtained as follows:
\vskip0.2cm

- Introducing/considering three particular lower triangular systems solved by Bernoulli numbers. The first two, which we may call {\it almost-even} and {\it almost-odd}, are introduced by exploiting a well known power series expansion involving Bernoulli polynomials. It is interesting to note that the coefficient matrices of such systems are particular submatrices of the l.t. Tartaglia matrix. The third one, the {\it almost-Ramanujan} system, is simply deduced from the 11 equations, solved by the absolute values of the first 11 B.n., listed by Ramanujan in the paper \cite{Ramanujan}.
\vskip0.1cm

- Noting that the almost-even, almost-odd, and almost-Ramanujan systems are structured in such a way that their coefficient matrices can be forced to be Toeplitz. This result follows, for the first two systems, from the matrix series representation of the Tartaglia matrix in terms of powers of a kind of {\it regularly weighted} lower shift matrix, and, for the third one, by a remarkable remark proved in the $11\times 11$ case, and conjectured in the general case.
\vskip0.1cm

- Proving that each of the three l.t.T. systems so obtained (even, odd and Ramanujan), which is solved by $\z$ (or $Z^T\z$), can be manipulated so to define a correspondent l.t.T. system whose solution is $Z^T\z$ (or $\z$).
\vskip0.2cm

The Ramanujan l.t. system in \cite{Ramanujan} has the remarkable peculiarity to have two null diagonals alternating the nonnull ones. The same peculiarity is inherited by its Toeplitz version, obtained in this paper (see (\ref{typeII}), (\ref{typeR})).
For some time we have tried to obtain by linear algebra arguments the system in \cite{Ramanujan} as a consequence of our odd and even systems, also with the aim to learn a technique for introducing a system possibly more sparse than and as simple as the Ramanujan one and, above all, its Toeplitz version. In order to do that, first of all it was necessary to nullify the second, the third, the fifth, the sixth, the eigth, the nineth, and so on, diagonals of our odd and even systems. At that time
we conceived the idea of a fast direct (not recursive) solver of l.t.T. systems.
In fact, the process of making null two diagonals every one, could
be repeated, so to finally transform the initial l.t.T. into the identity matrix.
Moreover, each step of such sort of Gaussian elimination procedure could be realized by
a left multiplication by a suitable l.t.T. matrix.
These remarks led us to conceive a $O(n\log_3 n)$ solver of l.t.T. systems
$A\x=\vf$ where $A$ is $n\times n$ with $n=3^s$, and then to extend the result,
obtaining analogous low complexity algorithms, ad hoc for the cases $n=b^s$, $b=2$ and $b$ generic.
Such exact algorithms are described in the present paper in detail, since we believe
that, for their not recursive character and for their clearness, they could be competitive with any known $O(n\log n)$ l.t.T. systems solver \cite{ltT_approxsolver_Bini}, \cite{ltT_approxsolver_BiniCalcolo}, \cite{ltT_approxsolver_Ng}, \cite{ltT_exactsolverCommengesMonsion}, \cite{ltT_exactsolver_BiniPan}, \cite{ltT_exactsolver_Morf}, \cite{ltT_solverMurphy}, \cite{ltT_TrenchInverse}.

In particular, as a first test, the $3^s\times 3^s$-algorithm could be applied to the Toeplitz versions (\ref{typeI}), (\ref{typeII}), (\ref{typeR}) of the Ramanujan system (\ref{ramanujan-system}) \cite{Ramanujan}, in order to compute the vector $\{\z\}_n$ that contains the first $n=3^s$ Bernoulli numbers in at most $O(n\log_3 n)$ a.o. (assuming already computed the entries of $A$ and $\vf$). Note that the first step of the algorithm can be in this case skipped, as it has been already performed explicitly by Ramanujan.

\section{Lower triangular Toeplitz matrices (l.t.T.)}

Let $Z$ be the following  $n\times n$ matrix
$$
    Z = \ba{ccccc}
           0 &  &  &   & \\
           1 &  &  &   & \\
             & 1&  &   & \\
             &  &\cdot& & \\
             &  &  & 1 & 0 \ea .
$$
$Z$ is usually called {\it lower-shift} due to the effect that its multiplication by a vector  $\v=[v_0\,v_1\,\cdots\,v_{n-1}]^T\in\CCn$ produces: $Z\v=[0\, v_0\, v_1\,\cdots\,v_{n-2}]^T$.
Let $\L$ be the subspace of $\CCnn$ of those matrices which commute with $Z$. It is simple to observe that $\L$ is a matrix algebra closed under inversion, that is if $A,B\in\L$ then $AB\in\L$ and if $A\in\L$ is nonsingular then $A^{-1}\in\L$. Let us investigate the structure of the matrices in $\L$. Let $A\in\CCnn$. Then
$$
    AZ = \ba{cccc}
           a_{12} & \cdot & a_{1n} & 0 \\
           \vdots &       & \vdots & \vdots \\
           a_{n2} & \cdot & a_{nn} & 0 \ea,\ \
    ZA = \ba{ccc}
            0     & \cdots & 0 \\
           a_{11} & \cdots & a_{1n}  \\
           \cdot &       & \cdot  \\
           a_{n-11} & \cdots & a_{n-1n} \ea .
$$
Forcing the equality between $AZ$ and $ZA$ we obtain the conditions
$a_{12}=a_{13}=\ldots a_{1n}=a_{2n}=\ldots a_{n-1,n}=0$ and
$a_{i,j+1}=a_{i-1,j}$, $i=2,\ldots,n$, $j=1,\ldots,n-1$, from which one deduces the structure of $A\in\L$: $A$ must be a {\it lower triangular Toeplitz} (l.t.T.) matrix, i.e. of the type
\begin{equation}\label{dots}
	A = \ba{ccccc}
           a_{11} &        &       &  & \\
           a_{21} & a_{11} &       &    & \\
           a_{31} & a_{21} & a_{11}&  & \\
           \cdot  &        & \cdot & \cdot &  \\
           a_{n1} & \cdot  & \cdot & a_{21} & a_{11} \ea .
\end{equation}
It follows that $\dim\L=n$ and that, by a well known general result \cite{Lancaster}, $\L$ can be represented as the set of all polynomials in $Z$, i.e. $\L=\{p(Z):\, p=\hbox{polynomials}\}$ .
Actually, by investigating the powers of $Z$ one realizes that the matrix $A$ in \eqref{dots} is exactly the polynomial $\sum_{k=1}^n a_{k1}Z^{k-1}$.

Note also that, as a consequence of the above arguments, the inverse of a l.t.T. matrix is still l.t.T., thus it is completely determined as soon as its first column is known.

In the next section we will illustrate an efficient algorithm for the solution of a lower triangular Toeplitz linear system  $A\x=\vf$, $A\in\L$, where $n=2^s$ . We will show that such operation can be realized trough $O(\log_2 n)$ matrix-vector products, where the matrices involved are l.t.T. and their dimension is $2^j\times 2^j$, with $j=2,\ldots,s$.
Since such products require no more than $c j2^j$ arithmetic operations (see Appendices A, B) the overall complexity of the proposed algorithm is
$O(n \log_2 n)$.


\section{An algorithm for the solution of a lower triangular Toeplitz linear system of $n$ equations, where $n$ is a power of $2$} 
\label{sec:algorithm}

In this section we present an algorithm of complexity $O(n \log_2 n)$ for the computation of $\x$ such that  $A\x=\vf$, where $A$ is a $n\times n$ lower triangular Toeplitz matrix, with $n$ power of $2$ and $[A]_{11}=1$.

\subsection{Preliminary Lemmas}

Given a vector $\v=[v_0\ v_1\ v_2\ \cdots\,]^T$, $v_i\in\CC$ (briefly $\v\in\CC^\NN$), let $L(\v)$ be the semi-infinite lower triangular Toeplitz matrix whose first column is $\v$, i.e.
$$
             L(\v) = \sum_{k=0}^{+\infty} v_k Z^k,\ \
             Z = \ba{cccc}
                   0 &   &       & \\
                   1 & 0 &       & \\
                     & 1 & 0     & \\
                     &   & \cdot & \cdot \ea .
$$
\begin{lemma}\label{lemma1}
	Let $\a$, $\b$, $\c$ be  vectors in $\CC^\NN$. Then
	$L(\a)L(\b)=L(\c)$ if and only if $L(\a)\b=\c$.
\end{lemma}
\begin{proof}
	If $L(\a)L(\b)=L(\c)$, then the first column of $L(\a)L(\b)$ must be equal to the first column of $L(\c)$, and these are the vectors $L(\a)\b$ and $\c$, respectively. Conversely, assume that $L(\a)\b=\c$ and consider the matrix $L(\a)L(\b)$. It is lower triangular Toeplitz being a product of lower triangular Toeplitz matrices, and, by hypothesis, its first column $L(\a)\b$ coincides with the vector $\c$, which in turn is the first column of the lower triangular Toeplitz matrix $L(\c)$. The thesis follows from the fact that l.t.T. matrices are uniquely defined by their first columns.
\end{proof}

Given a vector  $\v=[v_0\,v_1\,v_2\,\cdots]^T\in\CC^\NN$, let $E$ be the semi-infinite matrix with entries $0$ or $1$, which maps $\v$ into the vector
$E\v=[v_0\, 0\, v_1\, 0\, v_2\, 0\, \cdots]^T$:
$$
     E = \ba{cccc}
            1 &  &  &    \\
            0 &  &  &    \\
            0 & 1&  &    \\
            0 & 0&  &    \\
            0 & 0& 1&    \\
           \cdot & \cdot & \cdot & \cdot \ea.
$$
In other words, the application of $E$ to $\v$ has the effect of inserting a zero between two consecutive components of $\v$. It is easy to observe that
$$
    E^2 = \ba{cccc}
           1 &  &  &  \\
           0 &  &  &  \\
           0 &  &  &  \\
           0 &  &  &  \\
           0 & 1&  &  \\
           0 & 0&  &  \\
           0 & 0&  &  \\
           0 & 0&  &  \\
           0 & 0& 1&  \\
           \cdot & \cdot & \cdot & \cdot \ea,\ \
    E^s = \ba{cccc}
            1   &    &  &    \\
            \vn &    &  &    \\
            0   & 1  &  &    \\
            \vn & \vn&  &    \\
            0   & 0  & 1&    \\
           \cdot & \cdot & \cdot & \cdot \ea,\ \vn=\vn_{2^s-1},
$$
that is the application of  $E^s$ to $\v$ has the effect of inserting  $2^s-1$ zeros between two consecutive components of $\v$.
\begin{lemma}\label{lemma2}Let $\u$ and $\v$ be vectors in $\CC^\NN$ with $u_0=v_0=1$. Then
	$L(E\u)E\v = EL(\u)\v$, and, more in general, for each $s\in\NN$,
	$L(E^s\u)E^s\v = E^sL(\u)\v$.
\end{lemma}
\begin{proof}
	By inspecting the vectors $L(E\u)E\v$ and $EL(\u)\v$ one observes that they are equal. By multiplying $E$ on the left of the identity $L(E\u)E\v = EL(\u)\v$ and using the same identity also for the vectors $E\u$ and $E\v$, in place of $\u$ and $\v$ respectively, one observes that it also holds $L(E^2\u)E^2\v = E^2L(\u)\v$. And so on.
\end{proof}

\subsection{The algorithm}
Let $A$ be a $n\times n$ l.t.T. matrix, with $n$ power of $2$ and $[A]_{11}=1$. Assume we want to solve the system $A\x=\vf$. The algorithm presented below exploits the fact that $A^{-1}$ is still a $n\times n$ l.t.T. matrix.

\begin{itemize}
\item[1.]
Compute the first column of the l.t.T. matrix $A^{-1}$ by solving the particular linear system $A\x=\e_1$ via the algorithm \eqref{alg-prima-riga} of complexity $O(n\log_2 n)$ shown in the next section, based upon Lemmas \ref{lemma1}, \ref{lemma2} and their repeated application.
\item[2.]
Compute the l.t.T. matrix-vector product $A^{-1}\vf$ with no more than $O(n\log_2 n)$ arithmetic operations (see Appendices A and B).
\end{itemize}

\subsection{The computation of the first column of the inverse of a $n\times n$ l.t.T. matrix, where $n$ is a power of $2$}
For the sake of readability here we present the algorithm for the computation of $\x$ such that $A\x=\e_1$ in the particular case $n=8$. When suitable we briefly discuss the general case $n=2^s$, $s\in\NN$; nevertheless such case can be easily deduced from the considered one, and is reported in detail in Appendix C.

The algorithm consist of two parts. In the first one particular l.t.T. matrices are introduced and computed, with the property that their successive left multiplication by the matrix $A$ transforms $A$ into the the identity matrix.
In the second part such matrices are successively left multiplied by the vector $\e_1$.
As it will be clear throughout what follows, the method is nothing more than a kind of Gaussian elimination, where diagonals are nullified instead of columns. The overall cost of $O(n \log_2 n)$ comes from the fact that at each step of the first part a half of the remaining non null diagonals are nullified, and from the fact that in the second part the computations can be simplified by exploiting the structure of $\e_1$, which has only one nonzero component.

First of all observe that the $8\times 8$ matrix $A$ can be thought as the upper-left submatrix of a semi-infinite l.t.T. matrix $L(\a)$, whose first column is
$[1\, a_1\, a_2\,\cdot\, a_7\, a_8\,\cdot\,]^T$.\\[0.2cm]
{\bf Step 1.} Look for $\hat\a$ such that
$$
   L(\a)\hat\a
   = \ba{ccccccccc}
      1     &      &      &      &      &      &      &      &      \\
      a_1   & 1    &      &      &      &      &      &      &      \\
      a_2   & a_1  &  1   &      &      &      &      &      &      \\
      a_3   & a_2  & a_1  &  1   &      &      &      &      &      \\
      a_4   & a_3  & a_2  &  a_1 &  1   &      &      &      &      \\
      a_5   & a_4  & a_3  &  a_2 &  a_1 &   1  &      &      &      \\
      a_6   & a_5  & a_4  &  a_3 &  a_2 & a_1  &  1   &      &      \\
      a_7   & a_6  & a_5  &  a_4 &  a_3 & a_2  & a_1  &  1   &      \\
      \cdot & \cdot& \cdot& \cdot& \cdot& \cdot& \cdot& \cdot& \cdot\ea
     \ba{c}
      1 \\
      \hat a_1 \\
      \hat a_2 \\
      \hat a_3 \\
      \hat a_4 \\
      \hat a_5 \\
      \hat a_6 \\
      \hat a_7 \\
      \cdot \ea
 =  \ba{c}
      1 \\
      {\it 0} \\
      a_1^{(1)} \\
      {\it 0} \\
      a_2^{(1)} \\
      {\it 0} \\
      a_3^{(1)} \\
      {\it 0} \\
      \cdot \ea
= E\a^{(1)}
$$
for some $a_i^{(1)}\in\CC$, and compute such $a_i^{(1)}$. The computation of
$a_i^{(1)}$ requires, once $\hat\a$ is known, one l.t.T. $8\times 8$ ($2^s\times 2^s$)  matrix-vector product $-$ or, more precisely, two l.t.T. $4\times 4$ ($2^{s-1}\times 2^{s-1}$) matrix-vector products. We will see that $\hat \a$ is actually available with no computations.

Note that, due to Lemma \ref{lemma1}, we have  $L(\hat\a)L(\a)=L(E\a^{(1)})$, that is the l.t.T. matrix $L(\a)$ is transformed into a l.t.T. matrix which alternates to each nonnull diagonal a null one.\\[0.2cm]
{\bf Step 2.} Look for $\hat\a^{(1)}$ such that
$$
   L(E\a^{(1)})E\hat\a^{(1)}
   = \ba{ccccccccc}
      1     &      &      &      &      &      &      &      &      \\
      0     & 1    &      &      &      &      &      &      &      \\
      a_1^{(1)}   & 0    &  1   &      &      &      &      &      &      \\
      0     & a_1^{(1)}  &  0   &  1   &      &      &      &      &      \\
      a_2^{(1)}   & 0    & a_1^{(1)}  &  0   &  1   &      &      &      &      \\
      0     & a_2^{(1)}  & 0    &  a_1^{(1)} &  0   &   1  &      &      &      \\
   a_3^{(1)}   & 0    & a_2^{(1)}  &  0   &  a_1^{(1)} &   0  &  1   &      &      \\
   0     & a_3^{(1)}  & 0    &  a_2^{(1)} &  0   & a_1^{(1)}  &  0   &  1   &      \\
      \cdot & \cdot& \cdot& \cdot& \cdot& \cdot& \cdot& \cdot& \cdot\ea
     \ba{c}
      1 \\
      0 \\
      \hat a_1^{(1)} \\
      0 \\
      \hat a_2^{(1)} \\
      0 \\
      \hat a_3^{(1)} \\
      0 \\
      \cdot \ea
 =  \ba{c}
      1 \\
      0 \\
      {\it 0} \\
      0 \\
      a_1^{(2)} \\
      0 \\
      {\it 0} \\
      0 \\
      \cdot \ea
= E^2\a^{(2)}
$$
for some $a_i^{(2)}\in\CC$, and compute such $a_i^{(2)}$. The computation of
$a_i^{(2)}$ requires, once $\hat\a^{(1)}$ is known, one l.t.T. $4\times 4$ ($2^{s-1}\times 2^{s-1}$)  matrix-vector product $-$ or, more precisely, two l.t.T. $2\times 2$ ($2^{s-2}\times 2^{s-2}$) matrix-vector products.

Note that, due to Lemma \ref{lemma1}, we have  $L(E\hat\a^{(1)})L(E\a^{(1)})=L(E^2\a^{(2)})$, that is the l.t.T. matrix $L(\a)$ is transformed into a l.t.T. matrix which alternates to each nonnull diagonal three null ones.

Also note that, due to Lemma \ref{lemma2}, if $L(\a^{(1)})\hat\a^{(1)}=E\a^{(2)}$ then $L(E\a^{(1)})E\hat\a^{(1)}= E^2\a^{(2)}$. We will see that $\hat\a^{(1)}$ such that $L(\a^{(1)})\hat\a^{(1)}=E\a^{(2)}$ is actually available with no computations.\\[0.2cm]
{\bf Step 3.}  Look for  $\hat\a^{(2)}$ such that
$$
   L(E^2\a^{(2)})E^2\hat\a^{(2)}
   = \ba{ccccccccc}
      1     &      &      &      &      &      &      &      &      \\
      0     & 1    &      &      &      &      &      &      &      \\
      0     & 0    &  1   &      &      &      &      &      &      \\
      0     & 0    &  0   &  1   &      &      &      &      &      \\
      a_1^{(2)} & 0&  0   &  0   &  1   &      &      &      &      \\
      0     & a_1^{(2)}  & 0    &  0 &  0   &   1  &      &      &      \\
      0     & 0    & a_1^{(2)}  &  0   &  0 &   0  &  1   &      &      \\
      0     & 0  & 0    &  a_1^{(2)} &  0   & 0  &  0   &  1   &      \\
      \cdot & \cdot& \cdot& \cdot& \cdot& \cdot& \cdot& \cdot& \cdot\ea
     \ba{c}
      1 \\
      0 \\
      0 \\
      0 \\
      \hat a_1^{(2)} \\
      0 \\
      0 \\
      0 \\
      \cdot \ea
 =  \ba{c}
      1 \\
      0 \\
      0 \\
      0 \\
      {\it 0} \\
      0 \\
      0 \\
      0 \\
      \cdot \ea
= E^3\a^{(3)}
$$
for some $a_i^{(3)}\in\CC$, and compute such $a_i^{(3)}$. The computation of
$a_i^{(3)}$ requires, once $\hat\a^{(2)}$ is known, one l.t.T. $2\times 2$ ($2^{s-2}\times 2^{s-2}$)  matrix-vector product $-$ or, more precisely, two l.t.T. $1\times 1$ ($2^{s-3}\times 2^{s-3}$) matrix-vector products. That is, no operation in our case $n=8$, where no entry $a_i^{(3)}$, $i\geq 1$, is needed.

Note that, due to Lemma \ref{lemma1}, we have  $L(E^2\hat\a^{(2)})L(E^2\a^{(2)})=L(E^3\a^{(3)})$, that is the l.t.T. matrix $L(\a)$ is transformed into a l.t.T. matrix which alternates to each nonnull diagonal seven null ones.

Also note that, due to Lemma \ref{lemma2}, if $L(\a^{(2)})\hat\a^{(2)}=E\a^{(3)}$ then $L(E^2\a^{(2)})E^2\hat\a^{(2)}= E^3\a^{(3)}$. We will see that $\hat\a^{(2)}$ such that $L(\a^{(2)})\hat\a^{(2)}=E\a^{(3)}$ is actually available with no computations.\\[0.2cm]
Proceed this way, if $n=2^s>8$. Otherwise stop, the first part of the algorithm is complete.
\vskip0.2cm
Summarizing, we have proved that
\begin{equation}\label{alg-prima-riga}
	L(E^2\hat\a^{(2)}) L(E\hat\a^{(1)}) L(\hat\a)L(\a) = L(E^3\a^{(3)})
\end{equation}
where the upper left $8 \times 8$ submatrices of $L(\a)$ and of $L(E^3\a^{(3)})$ are the initial lower triangular Toeplitz matrix $A$ and the identity matrix, respectively:
$$
    L(\a) = \ba{cccccc}
             1   &     &     &    &   & \\
             a_1 & 1   &     &    &   & \\
            \cdot&\cdot&\cdot&    &   & \\
             a_7 &\cdot& a_1 & 1  &   & \\
             a_8 & a_7 &\cdot& a_1& 1 & \\
            \cdot&\cdot&\cdot&\cdot&\cdot&\cdot\ea,\ \
L(E^3\a^{(3)}) =  \ba{cccccc}
             1   &     &     &    &   & \\
             0   & 1   &     &    &   & \\
            \cdot&\cdot&\cdot&    &   & \\
             0   &\cdot& 0   & 1  &   & \\
             a_1^{(3)} & 0   &\cdot& 0  & 1 & \\
            \cdot&\cdot&\cdot&\cdot&\cdot&\cdot\ea \, .
$$
The operations we did so far are: $8\times 8$ l.t.T. $\cdot$ vector $+$
$4\times 4$ l.t.T. $\cdot$ vector
(if $A$ were $n\times n$ with $n=2^s$ the operations required would have been:  $2^s\times 2^s$ l.t.T. $\cdot$ vector $+\ldots+$
$4\times 4$ l.t.T. $\cdot$ vector).

Now let us move to our main purpose, compute the first column of $A^{-1}$, and thus let us show the second part of the algorithm.

Consider the following semi-infinite linear system:
\begin{equation}\label{semiinf-system}
	 L(\a) \z = E^2 \v
\end{equation}
where $\v$ is a generic semi-infinite vector in $\CC^\NN$
(if $A$ is $n\times n$ with $n=2^s$, then the matrix $E$ in \eqref{semiinf-system} must  be raised to the power $s-1$ rather than $2$).
Such system can be rewritten as follows
$$
     \ba{cc}
        A  &  O  \\
        \vdots & \ddots \ea
     \ba{c}
        \{\z\}_8 \\
          z_8 \\
          \cdot \ea
    = \ba{c}
         v_0 \\ 0 \\ 0 \\ 0 \\ v_1 \\ 0 \\ 0 \\ 0 \\ v_2 \\ \cdot \ea
$$
that is, pointing out the upper part of the system, consisting of only $8$ equations.
Before proceeding further, let us note that $\{\z\}_8$ is such that $A \{\z\}_8 = [ v_0 \ 0 \ 0 \ 0 \ v_1 \ 0 \ 0 \ 0 ]^T$, $v_0,v_1\in\CC$. Therefore the choices $v_0=1$ and $v_1=0$, would make $\{\z\}_8$ equal to  the vector we are looking for, $A^{-1}\e_1$.

By using the identity \eqref{alg-prima-riga} one immediately observes that the system $L(\a) \z = E^2 \v$ is equivalent to the following one
$$
  \ba{cc}
     I_8  &  O \\
     \vdots & \ddots \ea
  \ba{c}
     \{\z\}_8 \\
     \vdots \ea
  = L(E^3\a^{(3)}) \z
= L(\hat\a) L(E\hat\a^{(1)}) L(E^2\hat\a^{(2)}) E^2\v
$$
Due to Lemma \ref{lemma2} we can rewrite the right hand side in a more convenient way:
$$
  L(\hat\a) L(E\hat\a^{(1)}) L(E^2\hat\a^{(2)}) E^2\v
  = L(\hat\a) L(E\hat\a^{(1)}) E^2 L(\hat\a^{(2)}) \v
  = L(\hat\a) E L(\hat\a^{(1)}) E L(\hat\a^{(2)}) \v .
$$
Therefore, the following identity holds:
$$
  \ba{cc}
     I_8  &  O \\
     \vdots & \ddots \ea
  \ba{c}
     \{\z\}_8 \\
     \vdots \ea
  = L(\hat\a) E L(\hat\a^{(1)}) E L(\hat\a^{(2)}) \v .
$$
All the matrices involved on the right hand side are lower triangular. Moreover, the upper left square submatrices of $E$ of dimensions $8\times 8$, $4\times 4$ have half of its columns null, for example
$$
     \{E\}_4 = \ba{cc|cc}
                    1 & 0 & 0 & 0  \\
                    0 & 0 & 0 & 0  \\
                    0 & 1 & 0 & 0  \\
                    0 & 0 & 0 & 0  \ea ,\ \
     \{E\}_8 = \ba{cccc|cccc}
                    1 & 0 & 0 & 0 & 0 & 0 & 0 & 0 \\
                    0 & 0 & 0 & 0 & 0 & 0 & 0 & 0 \\
                    0 & 1 & 0 & 0 & 0 & 0 & 0 & 0 \\
                    0 & 0 & 0 & 0 & 0 & 0 & 0 & 0 \\
                    0 & 0 & 1 & 0 & 0 & 0 & 0 & 0 \\
                    0 & 0 & 0 & 0 & 0 & 0 & 0 & 0 \\
                    0 & 0 & 0 & 1 & 0 & 0 & 0 & 0 \\
                    0 & 0 & 0 & 0 & 0 & 0 & 0 & 0 \ea .
$$
These two observations let us obtain an effective representation of $\{\z\}_8$:
$$
  \{\z\}_8 = \{ L(\hat\a) \}_8 \{ E \}_8 \{ L(\hat\a^{(1)}) \}_8
              \{ E \}_8 \{ L(\hat\a^{(2)}) \}_8 \{ \v \}_8
           = \{ L(\hat\a) \}_8 \{ E \}_{8,4} \{ L(\hat\a^{(1)}) \}_4
              \{ E \}_{4,2} \{ L(\hat\a^{(2)}) \}_2 \{ \v \}_2.
$$
By using such formula, when $v_0=1$, $v_1=0$, the vector $\{\z\}_8$ can be computed by  performing the operations
$4\times 4$ l.t.T. $\cdot$ vector $+$
$8\times 8$ l.t.T. $\cdot$ vector
(if $A$ is $n\times n$ with $n=2^s$ the operations required would have been $4\times 4$ l.t.T. $\cdot$ vector  $+\ldots+$
$2^s\times 2^s$ l.t.T.$\cdot$ vector),
that is,  as many operations as the {\it Gaussian elimination}, the first part of the algorithm.

In conclusion, if $cj2^j$  is an upper bound for the cost of the product $2^j\times 2^j$ l.t.T. $\cdot$ vector, then the overall cost of the shown algorithm is $\tilde c\sum_{j=2}^s j2^j = O(s2^s) = O(n\log_2 n)$ for an $n \times n$ matrix $A$ with $n = 2^s$.

We still have to prove that the vector $\hat \a$ such that  $L(\a)\hat\a=E\a^{(1)}$ is indeed available with no computations.  To this aim it is sufficient to
observe that  {\small
\begin{equation}\label{hat_a}
	\ba{ccccccccccc}
	     1   &       &     &     &   &  &  & & & & \\
	     a_1 & 1     &     &     &   &  &  & & & & \\
	     a_2 & a_1   & 1   &     &   &  &  & & & & \\
	     a_3 & a_2   & a_1 & 1   &   &  &  & & & & \\
	     a_4 & a_3   & a_2 & a_1 & 1 &  &  & & & & \\
	     a_5 & a_4   & a_3 & a_2 & a_1 & 1 &  & & & &  \\
	     a_6 & a_5   & a_4 & a_3 & a_2 & a_1 & 1 &  & & & \\
	     a_7 & a_6 & a_5   & a_4 & a_3 & a_2 & a_1 & 1 & & & \\
	     a_8 & a_7 & a_6 & a_5   & a_4 & a_3 & a_2 & a_1 & 1 &  & \\
	     a_9 & a_8 & a_7 & a_6 & a_5   & a_4 & a_3 & a_2 & a_1 & 1 &  \\
 \cdot & \cdot & \cdot & \cdot & \cdot & \cdot & \cdot & \cdot & \cdot & \cdot & \cdot \ea
	    \ba{c}
	    1 \\ -a_1 \\ a_2 \\ -a_3 \\ a_4 \\ -a_5 \\ a_6 \\ -a_7 \\ a_8 \\ -a_9 \\ \cdot \ea
	   = \ba{c}
	        1 \\ 0 \\ 2a_2-a_1^2 \\ 0 \\
	        2a_4-2a_1a_3+a_2^2 \\ 0 \\
	        2a_6-2a_1a_5+2a_2a_4-a_3^2 \\ 0 \\
	        2a_8-2a_1a_7+2a_2a_6-2a_3a_5+a_4^2 \\ 0 \\ \cdot \ea ,
\end{equation}
$$
   L(\a) \big( \e_1+\sum_{i=1}^{+\infty} (-1)^i a_i\e_{i+1}\big)
       = \e_1+ \sum_{i=1}^{+\infty} \de_{i=0\mod 2}
              \big(2a_i + \sum_{j=1}^{i-1}(-1)^j a_j a_{i-j} \big)\e_{i+1} .
$$}
This can be verified by a direct calculation.

\subsection{Observations on the algorithm's core}
Given the vector $\a$ the problem of the computation of $\hat \a$ such that $L(\a)\hat \a = E\a^{(1)}$, for some $\a^{(1)}$ is indeed a polynomial arithmetic problem. In fact, due to Lemma \ref{lemma1}, the identity $L(\a)\hat \a = E\a^{(1)}$ is equivalent to the equality $L(\a)L(\hat \a) = L(E\a^{(1)})$, i.e.
$$
    (\, \sum_{k=0}^{+\infty} a_k Z^k \,)
    (\, \sum_{k=0}^{+\infty} \hat a_k Z^k \,)
  = \sum_{k=0}^{+\infty} a_k^{(1)} Z^{2k} .
$$
Therefore the polynomial arithmetic problem can be stated as follows:
\begin{itemize}
	\item[] given $a(z)=\sum_{k=0}^{+\infty} a_kz^k$, find a polynomial $\hat a(z)=\sum_{k=0}^{+\infty} \hat a_k z^k$ such that
	$$
	    \hat a(z) a(z) = a_0^{(1)}+a_1^{(1)}z^2+a_2^{(1)}z^4+\ldots =: a^{(1)}(z)
	$$
	for some coefficients $a_i^{(1)}$.
\end{itemize}
Such problem is a particular case of the more general problem: transform a \emph{full} polynomial $a(z)$ into a \emph{sparse} polynomial $a^{(1)}(z)=\sum_{k=0}^{+\infty} a_k^{(1)}z^{rk}$, for a fixed $r \in \NN$. It is possible to describe explicitly a polynomial $\hat a(z)$ that realizes such transformation, in fact the following theorem holds
\begin{theorem}\cite{BiniPC}\label{teo:polinomio}
	Given $a(z)=\sum_{k=0}^{+\infty} a_kz^k$, set
	$\hat a(z) = a(zt)a(zt^2)\cdots a(zt^{r-1})$ where $t$ is a  $r$-th principal root of the unity ($t\in\CC$, $t^r=1$, $t^i\neq 1$ for $0<i<r$). Then
	$$
	      \hat a(z) a(z) = a_0^{(1)}+a_1^{(1)}z^r+a_2^{(1)}z^{2r}+\ldots =: a^{(1)}(z)
	$$
	for some $a^{(1)}_i$. Moreover, if the coefficients of $a$ are real, then the coefficients of $\hat a$ are real.
\end{theorem}
Let us consider two Corollaries of such Theorem.
For $r=2$ we have  $\hat a(z)=a(-z)$, that is we regain the result \eqref{hat_a}.
It is clear that $a(-z)a(z)=a_0^{(1)}+a_1^{(1)}z^2+a_2^{(1)}z^4+\ldots$ (compare with \cite{Graeffe} and  the references therein). In this case the coefficients of $\hat a$ are available with no computations, we only need to compute the new coefficients  $a_i^{(1)}$.

For $r=3$ we have $\hat a(z)=a(zt)a(zt^2)$, $t=\exp{\ui \f{2\pi}{3}}$.
By Theorem \ref{teo:polinomio} the following equalities $a(z)a(zt)a(zt^2)=a_0^{(1)}+a_1^{(1)}z^3+a_2^{(1)}z^{6}+\ldots$ and
$$
      L(\hat \a)L(\a) = L(E\a^{(1)}),\ \
             E = \ba{cccc}
                   1 &  &  &  \\
                   0 &  &  &  \\
                   0 &  &  &  \\
                   0 & 1&  &  \\
                   0 & 0&  &  \\
                   0 & 0&  &  \\
                   0 & 0& 1&  \\
                   \cdot & \cdot & \cdot & \cdot \ea,
$$
hold, and the coefficients of $\hat a(z)=a(zt)a(zt^2)$ are real, provided that the ones of $a$ are.
This time, the coefficients of $\hat a$ are not easily readable from the coefficients of $a$. In order to compute them we observe that the polynomial equality $\hat a(z)=a(zt)a(zt^2)$ is equivalent to the matrix identity $L(\hat\a)=L(\c)L(\d)$, $c_k=a_kt^k$, $d_k=a_kt^{2k}$, and therefore we get the following formula
\begin{equation}\label{hat_a_3}
	  \hat\a = L(\c)\d = \Re[L(\c)]\Re[\d]-\Im[L(\c)]\Im[\d]
\end{equation}
where the last equality holds only if the coefficients of $a$ are real.

Later on we will describe an algorithm for the solution of systems $A\x=\vf$ where $A$ is  $3^s\times 3^s$ l.t.T. analogous to the one presented before, but using vectors $\hat\a$ such that the components in positions $2,3,5,6,8,9,\ldots$ of $L(\a)\hat\a$ are null. Thanks to Theorem \ref{teo:polinomio}, we have an explicit formula \eqref{hat_a_3} for such vectors, in terms of the product of a triangular Toeplitz matrix by a vector.

\section{Bernoulli numbers and triangular matrices}
\subsection{Bernoulli numbers and polynomials}

The conditions
$$
    B(x+1)-B(x) = nx^{n-1},\ \ \int_0^1 B(x)\,dx = 0,\ \ B(x)\ \hbox{polinomio}
$$
uniquely define the function $B(x)$. It is a particular degree $n$ monic polynomial called \emph{$n$-th Bernoulli polynomial} and usually denoted by the symbol $B_n(x)$.
It is simple to compute the first Bernoulli polynomials:
$$
   B_1(x)=x-\um,\ B_2(x)=x(x-1)+\f{1}{6},\ B_3(x)=x(x-\um)(x-1),\ \ldots\ .
$$
$B_0(x)$ is assumed equal to $1$.

It can be proved that Bernoulli polynomials define the coefficients of the power series representation of several functions, for instance to our aim it is useful to recall that the following power series expansion holds:
\begin{equation}\label{bernoulli_1}
	 \f{ t \exp{xt} }{ \exp{t}-1 } = \sum_{n=0}^{+\infty} \f{B_n(x)}{n!} t^n .
\end{equation}
Moreover, Bernoulli polynomials satisfy many identities. Among all we recall the following ones, concerning the value of their derivatives and their property of symmetry with respect to the line $x=\um$:
$$
       B_n^\prime(x) = n B_{n-1}(x),\ \ B_n(1-x) = (-1)^n B_n(x).
$$
It is simple to observe as a consequence of their definition and of the last identity that all the Bernoulli polynomials with odd degree (except $B_1(x)$) vanish for $x=0$. On the contrary, the value that an even degree Bernoulli polynomial attains in the origin is different from zero and especially important.
In particular, recall the following Euler formula
$$
       \zeta(2j) = \f{|B_{2j}(0)|(2\pi)^{2j}}{2(2j)!},\ \
       \zeta(s)=\sum_{k=1}^{+\infty} \f{1}{k^s},
$$
which shows the strict relation between the numbers $B_{2j}(0)$ and the values that the  Riemann Zeta function $\zeta(s)$ attains over all even positive integer numbers $2j$ \cite{BookOnBernoulli}, \cite{Bn_EulForm_Apostol}. For instance, from such relation and from the fact that  $\zeta(2j)\to 1$ if $j\to +\infty$, one deduces that $|B_{2j}(0)|$ tends to $+\infty$ almost the same way as $2(2j)!/(2\pi)^{2j}$ does.
Another important formula involving the values $B_{2j}(0)$ is the Euler-Maclaurin formula \cite{BookOnBernoulli}, which is useful for the computation of sums:
if $f$ is a smooth enough function over $[m,n]$, $m,n\in\ZZ$, then
\begin{equation}\label{E-M}
	\sum_{r=m}^n f(r) = \um [ f(m) + f(n) ] + \int_m^n f(x)\,dx
              + \sum_{j=1}^k \f{B_{2j}(0)}{(2j)!}[f^{(2j-1)}(n)-f^{(2j-1)}(m)]
                       + u_{k+1} ,
\end{equation}
where
\begin{align*}
	u_{k+1}&=\f{1}{(2k+1)!}\int_m^n f^{(2k+1)}(x)\o{B}_{2k+1}(x)\,dx \\
	 &=- \f{1}{(2k)!}\int_m^n f^{(2k)}(x)\o{B}_{2k}(x)\,dx\\
	&=\f{1}{(2k+2)!}\int_m^n f^{(2k+2)}(x)[B_{2k+2}(0)-\o{B}_{2k+2}(x)]\,dx
\end{align*}
 and $\o{B}_n$
is  $B_n|_{[0,1)}$ extended periodically over $\RR$. Let us recall that the
Eulero-Maclaurin formula also leads to an important representation of the error of the trapezoidal rule
$\I_h = h[\um g(a) + \sum_{r=1}^{n-1} g(a+rh) + \um g(b)]$, $h=\f{b-a}{n}$,
in the approximation of the definite integral $\I = \int_a^b g(x)\,dx$. Such representation, holding for functions $g$ which are smooth enough in $[a,b]$, is obtained by setting $m=0$ and $f(t)=g(a+th)$ in \eqref{E-M}:
\begin{equation}\label{T-R}
  \I_h= \I
   + \sum_{j=1}^k \f{h^{2j}B_{2j}(0)}{(2j)!} [g^{(2j-1)}(b)-g^{(2j-1)}(a)] + r_{k+1},\ \
   r_{k+1} = \f{g^{(2k+2)}(\xi)h^{2k+2}(b-a)B_{2k+2}(0)}{(2k+2)!},
\end{equation}
$\xi\in(a,b)$. Such representation of the error, in terms of even powers of $h$, shows the reason why the Romberg extrapolation method for estimating a definite integral is efficient, when combined with trapezoidal rule. From \eqref{T-R} it is indeed clear that $\tilde\I_{h/2}:=(2^2\I_{h/2}-\I_h)/(2^2-1)$ approximates $\I$ with an error of order $O(h^4)$, whereas the error made by  $\I_h$ and $\I_{h/2}$ is of order $O(h^2)$.

For these and many other reasons (see for instance \cite{MazurSteinBook}, \cite{Bn_MazurCyclotomicFields}, \cite{Bn_Mazur19onArsConjANDapplicinPureMath}, \cite{BookOnBernoulli}, \cite{Bn_EulForm_Apostol}), the values $B_{2j}(0)$ have their own name: \emph{Bernoulli numbers}.

\subsection{Bernoulli numbers solve triangular Toeplitz systems}
From \eqref{bernoulli_1} it follows that Bernoulli numbers satisfy the following identity
$$
       \f{ t }{ \exp{t}-1 } = - \um t
                   +   \sum_{k=0}^{+\infty} \f{B_{2k}(0)}{(2k)!} t^{2k} .
$$
Multiplying the latter by $\exp{t}-1$, expanding $\exp{t}$ in terms of powers of $t$, and setting to zero the coefficients of $t^i$ of the right hand side, $i=2,3,4,\ldots$, yields the following equations:
\begin{equation}\label{equazioni}
	- \um j + \sum_{k=0}^{[\f{j-1}{2}]} \cb{j}{2k} B_{2k}(0) = 0,\ \ j=2,3,4,\ldots\ .
\end{equation}
Now, putting together equations \eqref{equazioni} for $j$ even and for $j$ odd, we obtain two lower triangular linear systems that uniquely define Bernoulli numbers:
$$
      \ba{ccccc}
      \cb{2}{0} &           &           &           &  \\
      \cb{4}{0} & \cb{4}{2} &           &           &  \\
      \cb{6}{0} & \cb{6}{2} & \cb{6}{4} &           &  \\
      \cb{8}{0} & \cb{8}{2} & \cb{8}{4} & \cb{8}{6} &  \\
        \cdot   & \cdot     &  \cdot    & \cdot     & \cdot \ea
      \ba{c} B_0(0) \\ B_2(0) \\ B_4(0) \\ B_6(0) \\ \cdot \ea
   =  \ba{c} 1 \\ 2 \\ 3 \\ 4 \\ \cdot \ea ,
$$
$$
     \ba{ccccc}
        \cb{1}{0} &           &           &           &  \\
        \cb{3}{0} & \cb{3}{2} &           &           &  \\
        \cb{5}{0} & \cb{5}{2} & \cb{5}{4} &           &  \\
        \cb{7}{0} & \cb{7}{2} & \cb{7}{4} & \cb{7}{6} &  \\
        \cdot     &  \cdot    & \cdot     &   \cdot   & \cdot \ea
      \ba{c} B_0(0) \\ B_2(0) \\ B_4(0) \\ B_6(0) \\ \cdot \ea
   =  \ba{c} 1 \\ 3/2 \\ 5/2 \\ 7/2 \\ \cdot \ea .
$$
From such systems we can for instance easily compute the first Bernoulli numbers:
\begin{equation}\label{first9Bernnumb}
   1,\ \f{1}{6},\ -\f{1}{30},\ \f{1}{42},\ -\f{1}{30},
   \ \f{5}{66},\ -\f{691}{2730},\ \f{7}{6},\ -\f{3617}{510}.
\end{equation}
Now we want to obtain an analytic representation for the coefficients matrices $W_e$ and $W_o$ of such linear systems. To this end it is enough to observe that $W_e$ and $W_o$ are suitable submatrices of the Tartaglia matrix $X$, which can be represented as a power series. More precisely, set
$$
   Y = \ba{ccccc}
          0 &  &  &  & \\
          1 & 0&  &  & \\
            & 2& 0&  & \\
            &  & 3& 0& \\
            &  &  & \cdot& \cdot\ea , \ \
\phi = \ba{cccccc}
               0 &   &   &   &  &  \\
               2 &  0&   &   &  &  \\
                 & 12&  0&   &  &  \\
                 &   & 30&  0&  &  \\
                 &   &   & 56& 0&  \\
                 &   &   &   & \cdot & \cdot\ea ,\ \
2=1*2,\ 12=3*4,\ 30=5*6,\ \ldots\, ,
$$
and note that from the equality
{\small
$$
X : = \ba{cccccccc}
      \cb{0}{0} & & & & & & & \\
   \cb{{\it 1}}{{\it 0}} & \cb{1}{1} & & & & & & \\
   \cb{{\bf 2}}{{\bf 0}} & \cb{2}{1} & \cb{2}{2} & & & & & \\
\cb{{\it 3}}{{\it 0}} & \cb{3}{1} & \cb{{\it 3}}{{\it 2}} & \cb{3}{3} & & & & \\
\cb{{\bf 4}}{{\bf 0}} & \cb{4}{1} & \cb{{\bf 4}}{{\bf 2}} & \cb{4}{3} & \cb{4}{4} & & & \\
\cb{{\it 5}}{{\it 0}} & \cb{5}{1} & \cb{{\it 5}}{{\it 2}} & \cb{5}{3} & \cb{{\it 5}}{{\it 4}} & \cb{5}{5} & & \\
\cb{{\bf 6}}{{\bf 0}} & \cb{6}{1} & \cb{{\bf 6}}{{\bf 2}} & \cb{6}{3} & \cb{{\bf 6}}{{\bf 4}} & \cb{6}{5} & \cb{6}{6} & \\
\cdot & \cdot  & \cdot  & \cdot & \cdot  & \cdot & \cdot & \cdot \ea
= \ba{ccccccc}
1 &   &   &   &   &  & \\
1 & 1 &   &   &   &  & \\
1 & 2 & 1 &   &   &  & \\
1 & 3 & 3 & 1 &   &  & \\
1 & 4 & 6 & 4 & 1 &  & \\
1 & 5 & 10& 10& 5 & 1 & \\
\cdot & \cdot & \cdot & \cdot & \cdot & \cdot & \cdot \ea
=  \sum_{k=0}^{+\infty} \ \f{1}{k!}\ Y^k,
$$ }
which holds because $[X]_{ij} = \f{1}{(i-j)!}[Y^{i-j}]_{ij}
       = \f{1}{(i-j)!} j\cdots (i-2)(i-1) = \cb{i-1}{j-1}$, $1\leq j\leq i\leq n$,
it follows that
$$
      W_e = Z^T \phi \sum_{k=0}^{+\infty} \f{1}{(2k+2)!} \phi^k ,\ \ \
         W_o = \ba{ccccc}
              1 &    &    &    &  \\
                & 3 &    &    &  \\
                &    & 5 &    &  \\
                &    &    & 7 &  \\
                &    &    &    & \cdot \ea \cdot
    \sum_{k=0}^{+\infty} \f{1}{(2k+1)!} \phi^k .
$$
We can therefore rewrite the two linear systems solved by Bernoulli numbers as follows:
\begin{equation}\label{quasi-pari}
  \sum_{k=0}^{+\infty} \f{2}{(2k+2)!} \phi^k \b = \q^e,\ \
     \b = \ba{c} B_0(0) \\ B_2(0) \\ B_4(0) \\ B_6(0) \\ \cdot \ea,\
     \q^e = \ba{c} 1 \\ 1/3 \\ 1/5 \\ 1/7 \\ \cdot \ea ,
\end{equation}
\begin{equation}\label{quasi-dispari}
 \sum_{k=0}^{+\infty} \f{1}{(2k+1)!} \phi^k \b = \q^o,\ \
     \q^o = \ba{c} 1 \\ 1/2 \\ 1/2 \\ 1/2 \\ \cdot \ea .
\end{equation}
Now, let us show that systems \eqref{quasi-pari} and \eqref{quasi-dispari} are equivalent to two lower triangular Toeplitz linear systems. Our aim is to replace $\phi$, a matrix whose subdiagonal entries are all different, by a matrix whose subdiagonal entries are all equal.

Set $D=\diag(d_1,\,d_2,\,d_3,\,\ldots)$, $d_i\neq 0$. By investigating the nonzero entries of the matrix $D \phi D^{-1}$, it is easy to observe that it can be forced to be equal to a matrix of the form $x Z$; just choose $d_k=x^{k-1}d_1/(2k-2)!$,
$k=1,2,3,\ldots$. So, if
\begin{equation}\label{D}
 D = \ba{cccccc}
            1  &         &           &   &  &  \\
               &\f{x}{2!}&           &   &  &  \\
               &         &\f{x^2}{4!}&   &  &  \\
               &         &           &\cdot&  &  \\
               &         &           &   &\f{x^{n-1}}{(2n-2)!}&  \\
               &         &           &   &  & \cdot \ea ,
\end{equation}
we have the equality $D \phi D^{-1}=xZ$. Then from \eqref{quasi-pari} it follows
$$
D\q^e = \sum_{k=0}^{+\infty} \f{2}{(2k+2)!} D\phi^k D^{-1}D\b
= \sum_{k=0}^{+\infty} \f{2}{(2k+2)!} (D\phi D^{-1})^k D\b\, ,
$$
that is
\begin{equation}\label{pari}
  \sum_{k=0}^{+\infty} \f{2x^k}{(2k+2)!} Z^k D\b = D\q^e ,
\end{equation}
and analogously from \eqref{quasi-dispari} it follows
\begin{equation} \label{dispari}
 \sum_{k=0}^{+\infty} \f{x^k}{(2k+1)!} Z^k D\b = D\q^o .
\end{equation}
Summarizing, let $\z$ be the vector $D\b$. Then the vector $\{\b\}_n$ whose entries are the first $n$ Bernoulli numbers can be obtained by a two-phase procedure:
\begin{itemize}
 \item[1.]
compute the first $n$ components of the solution of the lower triangular Toeplitz system \eqref{pari} (or \eqref{dispari}), i.e. $\{\z\}_n$ such that
$\{\,\sum_{k=0}^{+\infty} \f{2x^k}{(2k+2)!} Z^k \,\}_n\{\z\}_n = \{D\q^e\}_n$
(or $\{\,\sum_{k=0}^{+\infty} \f{x^k}{(2k+1)!} Z^k \,\}_n\{\z\}_n = \{D\q^o\}_n$)
\item[2.] solve the linear system $\{D\}_n\{\b\}_n=\{\z\}_n$ over the rational field.
\end{itemize}

Observe that the computation in phase 1 can be done by means of the algorithm described alongside the previous section at a computational cost of $O(n\log_2 n)$, and that such algorithm can be made numerically stable by a suitable choice of the parameter $x$. For instance, the choice $x=(2\pi)^2$ would ensure the sequence $z_n=\f{x^{n-1}}{(2n-2)!}B_{2n-2}(0)$, $n\in\NN$, to be bounded; indeed in this case $|z_n|\to 2$ if $n\to+\infty$, due to Euler formula. So, in phase 1, one obtains $n$ machine numbers which are good approximations over $\RR$ of the quantities $x^sB_{2s}(0)/(2s)!$, $s=0,1,\ldots,n-1$. Then, in phase 2, one should find a way to deduce, from the machine numbers obtained, the rational Bernoulli numbers $B_{2s}(0)$, $s=0,1,\ldots,n-1$.

\subsection{The Ramanujan lower triangular Toeplitz linear system solved by Bernoulli numbers}
In \cite{Ramanujan} Ramanujan writes explicitly $11$ sparse equations solved by the $11$ Bernoulli numbers $B_2(0)$, $B_4(0)$, $\ldots$, $B_{22}(0)$. They are the first of an infinite set of sparse equations solved by all the Bernoulli numbers. The $11$ Ramanujan equations all together form a lower triangular system which, according to our notations and definitions, can be rewritten as follows
\begin{equation}\label{ramanujan-system}
      \ba{cccccccccccc}
   1 &  &  &  &  &  &  &  &  &  &  &  \\
   0 & 1&  &  &  &  &  &  &  &  &  &  \\
   0 & 0& 1&  &  &  &  &  &  &  &  &  \\
   \f{1}{3} & 0& 0& 1&  &  &  &  &  &  &  &  \\
   0 & \f{5}{2}& 0& 0& 1&  &  &  &  &  &  &  \\
   0 & 0& 11& 0& 0& 1&  &  &  &  &  &  \\
   \f{1}{5} & 0& 0& \f{143}{4}& 0& 0& 1&  &  &  &  &  \\
   0 & 4& 0& 0& \f{286}{3}& 0& 0& 1&  &  &  &  \\
   0 & 0& \f{204}{5}& 0& 0& 221& 0& 0& 1&  &  &  \\
   \f{1}{7} & 0& 0& \f{1938}{7}& 0& 0& \f{3230}{7}& 0& 0& 1&  &  \\
   0 & \f{11}{2}& 0& 0& \f{7106}{5}& 0& 0& \f{3553}{4}& 0& 0& 1&  \\
\cdot & \cdot & \cdot& \cdot& \cdot& \cdot&
        \cdot& \cdot& \cdot& \cdot& \cdot& \cdot\ea
     \ba{c} B_2(0) \\ B_4(0) \\ B_6(0) \\ B_8(0) \\ B_{10}(0) \\ B_{12}(0) \\
    B_{14}(0) \\ B_{16}(0) \\ B_{18}(0) \\ B_{20}(0) \\ B_{22}(0) \\ \cdot \ea
   = \ba{c} \f{1}{6} \\ -\f{1}{30} \\ \f{1}{42} \\ \f{1}{45} \\ -\f{1}{132} \\
      \f{4}{455} \\ \f{1}{120} \\ -\f{1}{306} \\ \f{3}{665} \\ \f{1}{231} \\
      -\f{1}{552} \\ \cdot \ea .
\end{equation}
For example, by using the last but three equations of such system, from the Bernoulli numbers $B_2(0),\ldots, B_{16}(0)$ listed in \eqref{first9Bernnumb}, the following further Bernoulli numbers can be easily obtained:
$$
          B_{18}(0) = \f{43867}{798}, \ \
          B_{20}(0) = - \f{174611}{330}, \ \
          B_{22}(0) = \f{854513}{138} .
$$
Let $R$ be the semi-infinite coefficient matrix of the above Ramanujan system. By recalling the definition of the semi-infinite lower shift matrix $Z$ and of the semi-infinite vector $\b=[B_0(0)\, B_2(0)\, B_4(0)\,\cdot\,]^T$, the Ramanujan system can be shortly indicated as $R (Z^T \b) = \vf$, where $\vf=[f_1\, f_2\, f_3\,\cdot\,]^T$ obviously denotes the right hand side vector in \eqref{ramanujan-system}.

Apparently the non-zero entries of $R$ are not related with each other, and it seems so also for the entries of $\vf$. That is, it seems to be not possible to guess, just by looking at the above $11$ equations, the twelfth equation of the Ramanujan system. We can only guess that the non-zero entries of $R$ are in the same positions as the non-zero entries of a lower triangular Toeplitz matrix $\tilde R$ of the form $\sum_{k=0}^{+\infty} w_k Z^{3k}$, and, may be, it is possible to guess the sign of the entries of $\vf$.

Actually it is not difficult to note that the following identity \textit{must hold}
\begin{equation}\label{conj}
      R \Lambda^{-1} = \Lambda^{-1} \tilde R,\ \ \
      \Lambda = Z^T D Z = \ba{cccc}
                           \f{x}{2!} &             &             &  \\
                                     & \f{x^2}{4!} &             &  \\
                                     &             & \f{x^3}{6!} &  \\
                                     &             &             & \cdot \ea
\end{equation}
where $D$ is defined in \eqref{D} and $\tilde R$ is the following lower triangular Toeplitz matrix:
$$
      \tilde R = \sum_{k=0}^{+\infty} \f{2x^{3k}}{(6k+2)!(2k+1)} Z^{3k}
 = \ba{cccccccccccc}
   1 &  &  &  &  &  &  &  &  &  &  &  \\
   0 & 1&  &  &  &  &  &  &  &  &  &  \\
   0 & 0& 1&  &  &  &  &  &  &  &  &  \\
   \f{2x^3}{8!3} & 0& 0& 1&  &  &  &  &  &  &  &  \\
   0 & \f{2x^3}{8!3}& 0& 0& 1&  &  &  &  &  &  &  \\
   0 & 0& \f{2x^3}{8!3}& 0& 0& 1&  &  &  &  &  &  \\
   \f{2x^6}{14!5} & 0& 0& \f{2x^3}{8!3}& 0& 0& 1&  &  &  &  &  \\
   0 & \f{2x^6}{14!5}& 0& 0& \f{2x^3}{8!3}& 0& 0& 1&  &  &  &  \\
   0 & 0& \f{2x^6}{14!5}& 0& 0& \f{2x^3}{8!3}& 0& 0& 1&  &  &  \\
   \f{2x^9}{20!7} & 0& 0& \f{2x^6}{14!5}& 0& 0& \f{2x^3}{8!3}& 0& 0& 1&  &  \\
   0 & \f{2x^9}{20!7}& 0& 0& \f{2x^6}{14!5}& 0& 0& \f{2x^3}{8!3}& 0& 0& 1&  \\
\cdot & \cdot & \cdot& \cdot& \cdot& \cdot&
        \cdot& \cdot& \cdot& \cdot& \cdot& \cdot\ea .
$$
In fact it is easy to check that the $11\times 11$ upper left submatrix of $R\Lambda^{-1}$ coincides with the $11\times 11$ upper left submatrix of $\Lambda^{-1}\tilde R$.

Assuming that the conjecture \eqref{conj} is true, we have that
$R (Z^T \b) = \vf$  iff
$R \Lambda^{-1}(\Lambda Z^T \b) = \vf$ iff
$\Lambda^{-1} \tilde R (Z^T D \b) = \vf$ iff
\begin{equation}\label{quasi-Ramanujan}
 \tilde R  (Z^T D \b) = Z^T D Z\vf.
\end{equation}
Thus, the vector $Z^T D\b$ solves a lower triangular Toeplitz system which is more sparse than the l.t.T. systems \eqref{pari}, \eqref{dispari}, since in its coefficient matrix two null diagonals alternate the nonnull ones. Such Ramanujan l.t.T. system will be defined more precisely in the following (see \eqref{typeII}, \eqref{typeR}).

\subsection{A unifying theorem with $6$ l.t.T. linear systems solved by Bernoulli numbers}
In this section we collect in a Theorem three l.t.T. linear systems solved by the vector $D_x\b$, say of type I, and the corresponding l.t.T. linear systems solved by the vector $Z^TD_x\b$, say of type II ($D_x$ is the matrix $D$ in \eqref{D}).
In fact, till now, we have only found two systems of type I, the even and odd systems \eqref{pari} and \eqref{dispari}, and, partially, one system of type II, the Ramanujan l.t.T. system \eqref{quasi-Ramanujan} (note that for the latter system only the coefficient matrix has been written explicitly).

In the following, first we state a Proposition which allows one to state a system of type II from a system of type I, and viceversa. Then we state the Theorem, with the six l.t.T. linear systems solved by Bernoulli numbers, and we prove it by applying the Proposition to the even, odd, and Ramanujan l.t.T. systems found till now, and, in the same time, by completing the definition of the Ramanujan l.t.T. system.

\begin{proposition}\label{pro:1}
Let $Z_{n-1}$ and $Z_n$ be the upper-left $(n-1)\times (n-1)$ and
$n\times n$ submatrices of the semi-infinite lower-shift matrix $Z$, respectively.
Assume that, for some $\al_j$ and $f_j$ (or $w_j$), the following equality holds: {\small
$$
    \sum_{j=0}^{n-1} \al_j Z_n^j
    \ba{c}
       B_0(0) \\
       \f{x}{2!}B_2(0) \\
       \f{x^2}{4!}B_4(0) \\
           \cdot \\
       \f{x^{n-1}}{(2(n-1))!}B_{2(n-1)}(0) \ea
   = \ba{c}
          \eta \\
       \f{x}{2!}f_1 \\
       \f{x^2}{4!}f_2 \\
           \cdot \\
       \f{x^{n-1}}{(2(n-1))!}f_{n-1} \ea
     + B_0(0) \ba{c}
                \mu \\
                 \al_1 \\
                 \al_2 \\
                 \cdot \\
                 \al_{n-1} \ea
   \ \ \Big(\ =  \ba{c}
        w_0 \\
       \f{x}{2!}w_1 \\
       \f{x^2}{4!}w_2 \\
           \cdot \\
       \f{x^{n-1}}{(2(n-1))!}w_{n-1} \ea\ \Big),
$$ }
where $\eta$ and $\mu$ are arbitrary parameters.
Then we have {\small
$$
    \sum_{j=0}^{n-2} \al_j Z_{n-1}^j
    \ba{c}
       \f{x}{2!}B_2(0) \\
       \f{x^2}{4!}B_4(0) \\
           \cdot \\
       \f{x^{n-1}}{(2(n-1))!}B_{2(n-1)}(0) \ea
   = \ba{c}
       \f{x}{2!}f_1 \\
       \f{x^2}{4!}f_2 \\
           \cdot \\
       \f{x^{n-1}}{(2(n-1))!}f_{n-1} \ea
   \ \ \Big(\ =  \ba{c}
       \f{x}{2!}w_1 \\
       \f{x^2}{4!}w_2 \\
           \cdot \\
       \f{x^{n-1}}{(2(n-1))!}w_{n-1} \ea
     - B_0(0) \ba{c}
                 \al_1 \\
                 \al_2 \\
                 \cdot \\
                 \al_{n-1} \ea\ \Big).
$$ }
Also the converse is true provided that
$\al_0 B_0(0)=\eta+B_0(0)\mu$
(or $\al_0 B_0(0) = w_0$).
\end{proposition}
\begin{proof} Exploit the equality $Z_n = \ba{cc} 0 & \vn^T \\ \e_1 & Z_{n-1}\ea$. The details are left to the reader.\end{proof}
\begin{theorem} Set
$$
     Z = \ba{cccc}
            0 &   &  &  \\
            1 & 0 &  &  \\
              & 1 & 0 &  \\
              &   & \cdot & \cdot \ea,\ \
    \a = \ba{c}
            a_0 \\
            a_1 \\
            a_2 \\
            \cdot \ea,\ \
    \ L(\a) = \sum_{i=0}^{+\infty}\, a_i\, Z^i
        = \ba{ccccc}
            a_0 &       &       &       & \\
            a_1 & a_0 &       &       & \\
            a_2 & a_1 & a_0 &       & \\
            a_3 & a_2 & a_1 & a_0 & \\
            \cdot & \cdot & \cdot & \cdot & \cdot \ea .
$$
Let $d(\z)$ be the diagonal matrix with $z_i$ as diagonal entries. Set
$$
 \b = [B_0(0) \ B_2(0) \ B_4(0) \ \cdot\ ]^T,\ \
     D_x = \diag(\f{x^i}{(2i)!},\, i=0,1,2,\ldots),\ x\in\RR,
$$
where $B_{2i}(0)$, $i=0,1,2,\ldots$, are the Bernoulli numbers.

\noindent
Then {\it the vectors $D_x\b$ and $Z^TD_x\b$ solve the following
l.t.T. linear systems}
\begin{equation}\label{typeI}
   L(\a)\,(D_x\b) = D_x\q,
\end{equation}
\begin{equation}\label{typeII}
   L(\a)\,(Z^TD_x\b) = d(\z)Z^TD_x\q,
\end{equation}
where the vectors $\a=(a_i)_{i=0}^{+\infty}$,
$\q=(q_i)_{i=0}^{+\infty}$, and $\z=(z_i)_{i=1}^{+\infty}$,
can assume respectively the values:
\begin{equation}\label{typeR}
  \bi{c}
   \displaystyle{
   a_i^R = \de_{i=0\mod 3} \f{2x^i}{(2i+2)!(\f{2}{3}i+1)},\ \
   q_i^R = \f{1}{(2i+1)(i+1)} (1 - \de_{i=2\mod 3}\f{3}{2}),\
   i=0,1,2,3,\ldots
    } \\
   \displaystyle{
   z_{i}^R = 1 - \de_{i=0\mod 3}\f{1}{\f{2}{3}i+1},\ i=1,2,3,\ldots,
    } \ei
\end{equation}
\begin{equation}\label{typeEven}
 \bi{c}
   \displaystyle{
   a_i^e = \f{2x^i}{(2i+2)!},\ \
   q_i^e = \f{1}{2i+1},\
   i=0,1,2,3,\ldots
   } \\
   \displaystyle{ z_{i}^e = \f{i}{i+1},\ i=1,2,3,\ldots,
    } \ei
\end{equation}
\begin{equation}\label{typeOdd}
  \bi{c}
   \displaystyle{
   a_i^o = \f{x^i}{(2i+1)!},\ i=0,1,2,3,\ldots,\
   q_0^o = 1,\ q_i^o =\f{1}{2},\ i=1,2,3,\ldots
   } \\
   \displaystyle{
   z_{i}^o = \f{2i-1}{2i+1},\ i=1,2,3,\ldots\ .
    } \ei
\end{equation}
\end{theorem}
\begin{proof}
From the Ramanujan semi-infinite l.t.T. linear system \eqref{quasi-Ramanujan}, we obtain the following finite linear system
\begin{equation}\label{star}
   \bi{l}
     \sum_{j=0}^{n-2} \al_j Z_{n-1}^j
     \ba{c}
       \f{x}{2!}B_2(0) \\
       \f{x^2}{4!}B_4(0) \\
           \cdot \\
       \f{x^{n-1}}{(2(n-1))!}B_{2(n-1)}(0) \ea
   = \ba{c}
       \f{x}{2!}f_1 \\
       \f{x^2}{4!}f_2 \\
           \cdot \\
       \f{x^{n-1}}{(2(n-1))!}f_{n-1} \ea,\
  \al_j = \de_{j=0\mod 3} \f{2x^j}{(2j+2)!(\f{2}{3}j+1)}, \\ \\
   f_1=\f{1}{6},\ f_2=-\f{1}{30},\ f_3=\f{1}{42},\ f_4=\f{1}{45},\  f_5=-\f{1}{132},\ f_6=\f{4}{455}, \\
   f_7=\f{1}{120},\ f_8=,-\f{1}{306}\ f_9=\f{3}{665},\ f_{10}=\f{1}{231},\ f_{11}=-\f{1}{552},\ldots\ . \ei
\end{equation}
Then, by Proposition \ref{pro:1}, if $\eta+B_0(0)\mu =\al_0 B_0(0)$, we have that {\small
$$
     \sum_{j=0}^{n-1} \al_j Z_n^j
         \ba{c}
       B_0(0)  \\
       \f{x}{2!}B_2(0) \\
       \f{x^2}{4!}B_4(0) \\
           \cdot \\
       \f{x^{n-1}}{(2(n-1))!}B_{2(n-1)}(0) \ea
   = \ba{c}
       \eta  \\
       \f{x}{2!}f_1 \\
       \f{x^2}{4!}f_2 \\
           \cdot \\
       \f{x^{n-1}}{(2(n-1))!}f_{n-1} \ea
     + B_0(0)
        \ba{c}
           \mu \\
            \al_1 \\
            \al_2 \\
            \cdot \\
            \al_{n-1} \ea,
$$}
or, more precisely, that {\small
$$
   (I + \f{2}{8!3}x^3 Z^3 + \f{2}{14!5}x^6 Z^6 +
     \f{2}{20!7}x^9 Z^9 +\ldots)
   \ba{c}
       B_0(0) \\
      \f{x}{2!} B_2(0) \\
      \f{x^2}{4!} B_4(0) \\
      \f{x^3}{6!} B_6(0) \\
      \f{x^4}{8!} B_8(0) \\
      \f{x^5}{10!} B_{10}(0) \\
      \f{x^6}{12!} B_{12}(0) \\
          \vdots \ea
  = \ba{c}
       1 \\
     \f{x}{2!}\f{1}{6} \\
     \f{x^2}{4!}(-\f{1}{30}) \\
     \f{x^3}{6!}\f{1}{42}+\f{2x^3}{8!3} \\
     \f{x^4}{8!}\f{1}{45} \\
     \f{x^5}{10!}(-\f{1}{132}) \\
     \f{x^6}{12!} \f{4}{455}+\f{2x^6}{14!5} \\
     \f{x^7}{14!} \f{1}{120} \\
     \f{x^8}{16!} (-\f{1}{306}) \\
     \f{x^9}{18!} \f{3}{665} + \f{2x^9}{20!7} \\
     \f{x^{10}}{20!} \f{1}{231} \\
     \f{x^{11}}{22!} (-\f{1}{552}) \\
        \vdots \ea
   = \ba{c}
       \f{1}{0!} ( \f{1}{1\cdot 1} ) \\
       \f{x}{2!} ( \f{1}{3\cdot 2} ) \\
       \f{x^2}{4!} ( \f{1}{5\cdot 3} - \f{1}{5\cdot 2} ) \\
       \f{x^3}{6!} ( \f{1}{7\cdot 4} ) \\
       \f{x^4}{8!} ( \f{1}{9\cdot 5} ) \\
       \f{x^5}{10!} ( \f{1}{11\cdot 6} - \f{1}{11\cdot 4} ) \\
       \f{x^6}{12!} ( \f{1}{13\cdot 7} ) \\
       \f{x^7}{14!} ( \f{1}{15\cdot 8} ) \\
       \f{x^8}{16!} ( \f{1}{17\cdot 9} - \f{1}{17\cdot 6} ) \\
       \f{x^9}{18!} ( \f{1}{19\cdot 10} ) \\
       \f{x^{10}}{20!} ( \f{1}{21\cdot 11} ) \\
       \f{x^{11}}{22!} ( \f{1}{23\cdot 12} -\f{1}{23\cdot 8} ) \\
        \vdots \ea .
$$ }
The latter equality is a clever remark that allows us to prove that $D_x\b$ must solve
the following Ramanujan l.t.T. system of type I:
\begin{equation}\label{ramanujan-system-def}
     \sum_{j=0}^{+\infty} \al_j Z^j D_x\b = D_x\q^R ,
\end{equation}
$$
 \al_j=\de_{j=0\mod 3}\f{2x^j}{(2j+2)!(\f{2}{3}j+1)},\
  q_j^R = \f{1}{(2j+1)(j+1)} ( 1-\de_{j=2\mod 3}\f{3}{2} ),\
   j=0,1,2,3,\ldots\ .
$$
Note that from the explicit expression of $\q^R$ just obtained, it follows
an explicit expression for the entries $f_i$ of the original Ramanujan system \eqref{ramanujan-system}, i.e.
$$
   f_i = \f{1}{(2i+1)(i+1)}
( 1 - \de_{i=2\mod 3}\f{3}{2} - \de_{i=0\mod 3}\f{ 1 }{ \f{2}{3}i+1 } ),\
   \ i=1,2,3,\ldots .
$$
Note also that \eqref{star} can be rewritten as
$$
    \sum_{j=0}^{n-2} \al_j Z_{n-1}^j I^2_n D_x\b
        = \diag(z_i,\,i=1,2,\ldots,n-1)I^2_n D_x\q^R
$$
for suitable $z_i$ (the meaning of $I^2_n$ is clear from the context).
Such $z_i$ are easily obtained by imposing the equality
$$
   (1-\de_{i=2\mod 3}\f{3}{2}-\de_{i=0\mod 3}\f{1}{\f{2}{3}i+1}) = z_i ( 1-\de_{i=2\mod 3}\f{3}{2} ),
$$
which leads to the formula:
$$
z_i = 1 - \f{ \de_{i=0\mod 3} \f{ 1 }{ \f{2}{3}i+1 } }{ 1-\de_{i=2\mod 3}\f{3}{2} }
    =  1 - \de_{i=0\mod 3} \f{ 1 }{ \f{2}{3}i+1 } .
$$
So, the l.t.T. type I and type II systems \eqref{typeI}, \eqref{typeII} and \eqref{typeR} hold.

Now let us consider the finite versions of the even and odd systems \eqref{pari} and \eqref{dispari},
$$
   \sum_{j=0}^{n-1} \f{2x^j}{(2j+2)!} Z_n^j I^1_n D_x\b = I^1_n D_x\q^e,\ \ \
   \sum_{j=0}^{n-1} \f{x^j}{(2j+1)!} Z_n^j I^1_n D_x\b = I^1_n D_x\q^o,
$$
and apply to them Proposition \ref{pro:1}: {\small
$$
     \sum_{j=0}^{n-2} \f{2x^j}{(2j+2)!} Z_{n-1}^j
         \ba{c}
       \f{x}{2!}B_2(0) \\
       \f{x^2}{4!}B_4(0) \\
           \cdot \\
       \f{x^{n-1}}{(2(n-1))!}B_{2(n-1)}(0) \ea
=    \ba{c}
       \f{x}{2!} \f{1}{3} \\
       \f{x^2}{4!} \f{1}{5} \\
           \cdot \\
       \f{x^{n-1}}{(2(n-1))!} \f{1}{2n-1} \ea
   - B_0(0) \ba{c}
              \f{2x}{4!} \\
              \f{2x^2}{6!} \\
              \cdot  \\
              \f{2x^{n-1}}{(2n)!} \ea
  =  \ba{c}
           x \f{2}{4!} \\
           x^2 \f{4}{6!} \\
           x^3 \f{6}{8!} \\
           \cdot\\
           x^{n-1} \f{2(n-1)}{(2n)!} \ea ,
$$
$$
     \sum_{j=0}^{n-2} \f{x^j}{(2j+1)!} Z_{n-1}^j
         \ba{c}
       \f{x}{2!}B_2(0) \\
       \f{x^2}{4!}B_4(0) \\
           \cdot \\
       \f{x^{n-1}}{(2(n-1))!}B_{2(n-1)}(0) \ea
=   \ba{c}
       \f{x}{2!} \f{1}{2} \\
       \f{x^2}{4!} \f{1}{2} \\
           \cdot \\
       \f{x^{n-1}}{(2(n-1))!} \f{1}{2} \ea
   - B_0(0) \ba{c}
              \f{x}{3!} \\
              \f{x^2}{5!} \\
              \cdot  \\
              \f{x^{n-1}}{(2n-1)!} \ea
  =  \ba{c}
           x \f{1}{3!2} \\
           x^2 \f{3}{5!2} \\
           x^3 \f{5}{7!2} \\
           \cdot\\
           x^{n-1} \f{2n-3}{(2n-1)!2} \ea .
$$ }
From the above identities it follows that{\small
$$
     \sum_{j=0}^{n-2} \f{2x^j}{(2j+2)!} Z_{n-1}^j
         \ba{c}
       \f{x}{2!}B_2(0) \\
       \f{x^2}{4!}B_4(0) \\
           \cdot \\
       \f{x^{n-1}}{(2(n-1))!}B_{2(n-1)}(0) \ea
   = 2\ba{c}
       \f{x}{2!} \f{1}{4\cdot 3} \\
       \f{x^2}{4!} \f{2}{6\cdot 5} \\
       \f{x^3}{6!} \f{3}{8\cdot 7} \\
           \cdot \\
       \f{x^{n-1}}{(2(n-1))!} \f{n-1}{2n(2n-1)} \ea
   = \diag(\f{i}{i+1},i=1\ldots n-1) I^2_{n}  D_x\q^e  ,
$$
$$
     \sum_{j=0}^{n-2} \f{x^j}{(2j+1)!} Z_{n-1}^j
         \ba{c}
       \f{x}{2!}B_2(0) \\
       \f{x^2}{4!}B_4(0) \\
           \cdot \\
       \f{x^{n-1}}{(2(n-1))!}B_{2(n-1)}(0) \ea
=  \ba{c}
       \f{x}{2!} \f{1}{2\cdot 3} \\
       \f{x^2}{4!} \f{3}{2\cdot 5} \\
       \f{x^3}{6!} \f{5}{2\cdot 7} \\
           \cdot \\
       \f{x^{n-1}}{(2(n-1))!} \f{2n-3}{2(2n-1)} \ea
   = \diag(\f{2i-1}{2i+1},i=1\ldots n-1) I^2_{n} D_x\q^o .
$$ }
So, also even and odd type II linear systems \eqref{typeII}, \eqref{typeEven} and \eqref{typeOdd} hold.
\end{proof}

\subsection{On the need of a new algorithm for the solution of l.t.T. linear systems}

Now it is clear that the first $n$ Bernoulli numbers $b_i$, unless the factors $(D_x)_{ii}$, solve lower triangular Toeplitz systems $A\x = \vf$, where $A$ is the $n\times n$ upper left submatrix of the semi-infinite matrix $L(\a)$ in \eqref{typeI} (or \eqref{typeII}). Of course one can compute the $(D_x\b)_i$ via the algorithm described in Section \ref{sec:algorithm}, well defined for $n=2^s$.

By representing the first column of the lower triangular Toeplitz matrix $A$ in a row, the first part of such algorithm, i.e. the part in which $A$ is transformed into the identity matrix, can be schematized through the following steps:
$$
   \bi{cccccccccccccccccccc}
   1&*&*&*&*&*&*&*&*&*&*&*&*&*&*&*&*&*&\cdot&\ \to \\
   1&{\it 0}&*&{\it 0}&*&{\it 0}&*&{\it 0}&*&{\it 0}&*&{\it 0}&*&{\it 0}&*&{\it 0}&*&{\it 0}&\cdot&\ \to \\
   1&0&{\it 0}&0&*&0&{\it 0}&0&*&0&{\it 0}&0&*&0&{\it 0}&0&*&0&\cdot&\ \to \\
   1&0&0&0&{\it 0}&0&0&0&*&0&0&0&{\it 0}&0&0&0&*&0&\cdot&\ \to \\
   1&0&0&0&0&0&0&0&{\it 0}&0&0&0&0&0&0&0&*&0&\cdot&\ \to \\
    \cdot &\cdot & \cdot& \cdot& \cdot& \cdot& \cdot& \cdot& \cdot& \cdot& \cdot& \cdot& \cdot& \cdot& \cdot& \cdot& \cdot& \cdot& \cdot& \cdot \ei  \eqno(O(n\log_2 n))
$$
(four steps if, for example, $n=16$).

It is clear that the algorithm works very well if applied to the even and odd type I and type II l.t.T. systems, but it does not appear the best possible algorithm to solve the Ramanujan type I and type II l.t.T. linear systems, for instance the system \eqref{ramanujan-system-def}. A better algorithm would clearly be one whose first part could be schematically represented as follows:
{\small
$$
   \bi{ccccccccccccccccccccccccccccccc}
   1&*&*&*&*&*&*&*&*&*&*&*&*&*&*&*&*&*&*&*&*&*&*&*&*&*&*&*&*&\cdot&\to \\
   1&{\it 0}&{\it 0}&*&{\it 0}&{\it 0}&*&{\it 0}&{\it 0}&*&{\it 0}&{\it 0}&*&{\it 0}&{\it 0}&*&{\it 0}&{\it 0}&*&{\it 0}&{\it 0}&*&{\it 0}&{\it 0}&*&{\it 0}&{\it 0}&*&{\it 0}&\cdot&\to \\
   1&0&0&{\it 0}&0&0&{\it 0}&0&0&*&0&0&{\it 0}&0&0&{\it 0}&0&0&*&0&0&{\it 0}&0&0&{\it 0}&0&0&*&0&\cdot&\to \\
   1&0&0&0&0&0&0&0&0&{\it 0}&0&0&0&0&0&0&0&0&{\it 0}&0&0&0&0&0&0&0&0&*&0&\cdot&\to \\
   \cdot& \cdot & \cdot& \cdot& \cdot& \cdot& \cdot& \cdot& \cdot& \cdot& \cdot& \cdot& \cdot& \cdot& \cdot& \cdot& \cdot& \cdot& \cdot& \cdot&   \cdot   \ei\eqno(O(n\log_3 n))
$$}
(three steps if, for example, $n=27$). In other words, each step would consist of nullifying $2/3$ of the still remaining nonzero diagonals, instead of nullifying half of them. Such algorithm, moreover, would require one step less when applied to the Ramanujan l.t.T. linear systems \eqref{typeI}, \eqref{typeII}, \eqref{typeR}, \eqref{ramanujan-system-def}.

Now, it is possible to introduce such algorithm, well defined for $n=3^s$; it is presented in the following last section of this work. Then, in Appendix C, a general l.t.T. linear system solver is described, well defined when $n=b^s$, which includes the previously shown cases $n=2^s$, $n=3^s$ as particular cases.

\section{An algorithm for solving a system of $n$ linear lower triangular Toeplitz equations, with $n$ power of $3$}

In this section it is shown an algorithm which computes $\x$ such that $A\x=\vf$, being $A$ a lower triangular $n\times n$ Toeplitz matrix with $n$ power of $3$ and $[A]_{11}=1$. Its computational cost is $O(n\log_3 n)$.

We need to rewrite Lemma \ref{lemma2} in a version suitable for the case $n$ power of $3$.

Given a vector $\v=[v_0\,v_1\,v_2\,\cdot\,]^T\in\CC^\NN$, let $E$ be the semi-infinite $0$-$1$ matrix which maps $\v$ into the vector
$E\v=[v_0\, 0\, 0\, v_1\, 0\, 0\, v_2\, 0\, 0\, \cdot\,]^T$:
$$
     E = \ba{cccc}
            1 &  &  &    \\
            0 &  &  &    \\
            0 &  &  &    \\
            0 & 1&  &    \\
            0 & 0&  &    \\
            0 & 0&  &    \\
            0 & 0& 1&    \\
           \cdot & \cdot & \cdot & \cdot \ea.
$$
In other words, the action of $E$ over $\v$ has the effect of introducing two zeros between two successive components of $\v$. Observe that
$$
    E^2 = \ba{cccc}
           1 &  &  &  \\
           0 &  &  &  \\
           0 &  &  &  \\
           0 &  &  &  \\
           0 &  &  &  \\
           0 &  &  &  \\
           0 &  &  &  \\
           0 &  &  &  \\
           0 &  &  &  \\
           0 & 1&  &  \\
           \cdot & \cdot & \cdot & \cdot \ea,\ \
    E^s = \ba{cccc}
            1   &    &  &    \\
            \vn &    &  &    \\
            0   & 1  &  &    \\
            \vn & \vn&  &    \\
            0   & 0  & 1&    \\
           \cdot & \cdot & \cdot & \cdot \ea,\ \vn=\vn_{3^s-1},
$$
that is, the action of  $E^s$ over $\v$ has the effect of introducing $3^s-1$ zeros between two successive components of $\v$.

\begin{lemma}\label{lemma2_npower3}Let $\u$, $\v$ be vectors of $\CC^\NN$ with $u_0=v_0=1$. Then $L(E\u)E\v = EL(\u)\v$, and, more in general, for any $s\in\NN$ we have $L(E^s\u)E^s\v = E^sL(\u)\v$.
\end{lemma}
\begin{proof}Proceed as in the case $n$ power of $2$.
\end{proof}

\subsection{The algorithm}
Let $A$ be a l.t.T. $n \times n$ matrix with $n$ power of $3$ and $[A]_{11}=1$. We want to solve the system $A\x=\vf$ by a procedure consisting of two parts:
\begin{itemize}
\item[1]
Compute the first column of the l.t.T. $n\times n$ matrix $A^{-1}$, i.e. solve the particular l.t.T. linear system $A\x=\e_1$ by using the algorithm of computational cost $O(n\log_3 n)$ shown in the following section, based upon the successive application of Lemmas \ref{lemma1} and \ref{lemma2_npower3}
\item[2]
Compute the l.t.T. matrix vector product $A^{-1}\vf$ performing no more than $O(n\log_3 n)$ arithmetic operations (see Appendices A and B).
\end{itemize}

\subsection{The computation of the first column of the inverse of a l.t.T. $n\times n$ matrix with $n$ power of $3$}

For the sake of simplicity let us present the algorithm for the computation of $\x$ such that $A\x = \e_1$ when $n=9$, underlining, however, what are the significant changes in the general case $n=3^s$, $s\in \NN$. See the Appendix C, if interested in the details of the general case. The algorithm is similar to the one shown for $n$ power of $2$. The overall cost $O(n\log_3 n)$ of the algorithm comes from the fact that, at each step of the first part, $2/3$ of the nonzero diagonals are nullified, and from the fact that the second part can be simplified by noting that the vector $\e_1$ has only one nonzero component.

First of all observe that the $9\times 9$ matrix $A$ can be seen as the upper left  submatrix of a semi-infinite lower triangular Toeplitz matrix $L(\a)$ whose first column is
$[1\, a_1\, a_2\,\cdot\, a_7\, a_8\, a_9\,\cdot\,]^T$.\\[0.2cm]
Step 1. Find $\hat\a$ such that
$$
   L(\a)\hat\a
   = \ba{cccccccccc}
      1     &      &      &      &      &      &      &      &      & \\
      a_1   & 1    &      &      &      &      &      &      &      & \\
      a_2   & a_1  &  1   &      &      &      &      &      &      & \\
      a_3   & a_2  & a_1  &  1   &      &      &      &      &      & \\
      a_4   & a_3  & a_2  &  a_1 &  1   &      &      &      &      & \\
      a_5   & a_4  & a_3  &  a_2 &  a_1 &   1  &      &      &      & \\
      a_6   & a_5  & a_4  &  a_3 &  a_2 & a_1  &  1   &      &      & \\
      a_7   & a_6  & a_5  &  a_4 &  a_3 & a_2  & a_1  &  1   &      & \\
      a_8   & a_7  & a_6  & a_5  &  a_4 &  a_3 & a_2  & a_1  &  1   &      \\
      \cdot & \cdot& \cdot& \cdot& \cdot& \cdot& \cdot& \cdot& \cdot\ea
     \ba{c}
      1 \\
      \hat a_1 \\
      \hat a_2 \\
      \hat a_3 \\
      \hat a_4 \\
      \hat a_5 \\
      \hat a_6 \\
      \hat a_7 \\
      \hat a_8 \\
      \cdot \ea
 =  \ba{c}
      1 \\
      {\it 0} \\
      {\it 0} \\
      a_1^{(1)} \\
      {\it 0} \\
      {\it 0} \\
      a_2^{(1)} \\
      {\it 0} \\
      {\it 0} \\
      \cdot \ea
= E\a^{(1)}
$$
for some $a_i^{(1)}\in\CC$ and compute such $a_i^{(1)}$. The computation of
$a_i^{(1)}$ requires, once $\hat\a$ is known, one $9\times 9$ ($3^s\times 3^s$) l.t.T. matrix vector product -- or, more precisely, three $3\times 3$ ($3^{s-1}\times 3^{s-1}$) l.t.T. matrix vector products; the computation of $\hat\a$ requires one $9\times 9$ ($3^s\times 3^s$) l.t.T matrix vector product (see \eqref{hat_a_3}).
\vskip0.1cm

Note that, due to Lemma \ref{lemma1} we have then that $L(\hat\a)L(\a)=L(E\a^{(1)})$, that is the l.t.T. matrix $L(\a)$ is transformed into a l.t.T. matrix which alternates to each nonzero diagonal two null diagonals.\\[0.2cm]
Step 2. Find $\hat\a^{(1)}$ such that
$$
   L(E\a^{(1)})E\hat\a^{(1)}
   = \ba{cccccccccc}
      1     &      &      &      &      &      &      &      &     &  \\
      0     & 1    &      &      &      &      &      &      &     &  \\
      0     & 0    &  1   &      &      &      &      &      &     &  \\
      a_1^{(1)}& 0 &  0   &   1  &      &      &      &      &     &  \\
      0     & a_1^{(1)}& 0&   0  &   1  &      &      &      &     &  \\
      0     &   0  & a_1^{(1)}& 0&   0  &   1  &      &      &     &  \\
      a_2^{(1)}& 0 & 0    & a_1^{(1)}& 0&   0  &   1  &      &     &  \\
      0     & a_2^{(1)}& 0&  0   & a_1^{(1)}& 0&   0  &   1  &     &  \\
      0     & 0    & a_2^{(1)}& 0&  0   & a_1^{(1)}& 0&   0  &   1 &  \\
      \cdot & \cdot& \cdot& \cdot& \cdot& \cdot& \cdot& \cdot& \cdot & \cdot \ea
     \ba{c}
      1 \\
      0 \\
      0 \\
      \hat a_1^{(1)} \\
      0 \\
      0 \\
      \hat a_2^{(1)} \\
      0 \\
      0 \\
      \cdot \ea
 =  \ba{c}
      1 \\
      0 \\
      0 \\
      {\it 0} \\
      0 \\
      0 \\
      {\it 0} \\
      0 \\
      0 \\
      \cdot \ea
= E^2\a^{(2)}
$$
for some $a_i^{(2)}\in\CC$ and compute such $a_i^{(2)}$. The computation of $a_i^{(2)}$ requires, once $\hat\a^{(1)}$ is known, one $3\times 3$ ($3^{s-1}\times 3^{s-1}$) l.t.T. matrix vector product -- or, more precisely, three $1\times 1$ ($3^{s-2}\times 3^{s-2}$) l.t.T. matrix vector products.
That is, no operation in our case $n=9$, where no entry $a_i^{(2)}$, $i\geq 1$, is needed.
\vskip0.1cm

Note that due to Lemma \ref{lemma1}, we have that  $L(E\hat\a^{(1)})L(E\a^{(1)})=L(E^2\a^{(2)})$,
i.e. the l.t.T. matrix  $L(\a)$ is transformed into a l.t.T. matrix which alternates to each nonzero diagonal eight null diagonals.
\vskip0.1cm

Also note that, due to Lemma \ref{lemma2_npower3}, if $L(\a^{(1)})\hat\a^{(1)}=E\a^{(2)}$ then $L(E\a^{(1)})E\hat\a^{(1)}= E^2\a^{(2)}$. The computation of $\hat\a^{(1)}$ such that
$L(\a^{(1)})\hat\a^{(1)}=E\a^{(2)}$
requires one $3\times 3$ ($3^{s-1}\times 3^{s-1}$) l.t.T. matrix vector product (see \eqref{hat_a_3}).
\vskip0.2cm

Proceed this way, if $n=3^s>9$. Otherwise the first part of the algorithm is complete.
\vskip0.2cm

Summarizing, we have shown that,
\begin{equation}\label{alg-prima-riga-npower3}
  L(E\hat\a^{(1)})L(\hat\a)L(\a)=L(E^2\a^{(2)})
\end{equation}
where the upper left $9\times 9$ submatrices of $L(\a)$ and of $L(E^2\a^{(2)})$ are the lower triangular Toeplitz matrix initially given $A$ and the identity matrix $I$, respectively,
$$
    L(\a) = \ba{ccccccc}
             1   &     &     &    &   & & \\
             a_1 & 1   &     &    &   & & \\
            \cdot&\cdot&\cdot&    &   & & \\
             a_7 &\cdot& a_1 & 1  &   & & \\
             a_8 & a_7 &\cdot& a_1& 1 & & \\
             a_9 & a_8 & a_7 &\cdot& a_1& 1 & \\
            \cdot&\cdot&\cdot&\cdot&\cdot&\cdot&\cdot\ea,\ \
L(E^2\a^{(2)}) =  \ba{cccccc}
             1   &     &     &    &   & \\
             0   & 1   &     &    &   & \\
            \cdot&\cdot&\cdot&    &   & \\
             0   &\cdot& 0   & 1  &   & \\
             a_1^{(2)} & 0   &\cdot& 0  & 1 & \\
            \cdot&\cdot&\cdot&\cdot&\cdot&\cdot\ea,
$$
and the operations we did so far are: two products $9\times 9$ l.t.T. matrix $\cdot$ vector $+$ one product $3\times 3$ l.t.T. matrix $\cdot$ vector (if $A$ were $n\times n$ with $n=3^s$ the operations required would have been: two products $3^s\times 3^s$ l.t.T. matrix $\cdot$ vector $+\ldots+$ two products $9\times 9$ l.t.T. matrix $\cdot$ vector $+$ one product $3\times 3$ l.t.T. matrix $\cdot$ vector).

Now let us move to our purpose, compute the first column of $A^{-1}$, and thus let us show the second part of the algorithm.
Consider the following semi-infinite linear system:
\begin{equation}\label{semi-infinite-npower3}
     L(\a) \z = E \v
\end{equation}
where $\v$ is a generic semi-infinite vector in $\CC^\NN$
(if $A$ is $n\times n$ with $n = 3^s$, then the matrix $E$ in \eqref{semi-infinite-npower3} must be raised to the power $s-1$). Such system can be rewritten as follows
$$
     \ba{cc}
        A  &  O  \\
        \vdots & \ddots \ea
     \ba{c}
        \{\z\}_9 \\
          z_9 \\
          \cdot \ea
    = \ba{c}
         v_0 \\ 0 \\ 0 \\ v_1 \\ 0 \\ 0 \\ v_2 \\ 0 \\ 0 \\ v_3 \\ \cdot \ea ,
$$
that is, pointing out the upper part of the system, consisting of only $9$ equations.
Before proceeding further, let us note that $\{\z\}_9$ is such that $A \{\z\}_9 = [ v_0 \ 0 \ 0  \ v_1 \ 0 \ 0 \ v_2 \ 0 \ 0]^T$, $v_0,v_1,v_2\in\CC$. Therefore the choices $v_0=1$ and $v_1=v_2=0$, would make $\{\z\}_9$ equals to the vector we are looking for, $A^{-1}\e_1$.

By using the identity \eqref{alg-prima-riga-npower3} one immediately observes that the system $L(\a) \z = E \v$ is equivalent to the following one:
$$
  \ba{cc}
     I_9  &  O \\
     \vdots & \ddots \ea
  \ba{c}
     \{\z\}_9 \\
     \vdots \ea
  = L(E^2\a^{(2)}) \z
= L(\hat\a) L(E\hat\a^{(1)}) E\v .
$$
Due to Lemma \ref{lemma2_npower3} we can rewrite the right hand side in a more convenient way:
$$
  L(\hat\a) L(E\hat\a^{(1)}) E\v
  = L(\hat\a) E L(\hat\a^{(1)}) \v .
$$
Therefore, the following identity holds:
$$
  \ba{cc}
     I_9  &  O \\
     \vdots & \ddots \ea
  \ba{c}
     \{\z\}_9 \\
     \vdots \ea
  = L(\hat\a) E L(\hat\a^{(1)}) \v .
$$
The matrices involved on the right hand side are all lower triangular. Moreover the upper left square submatrices of E of dimensions $9\times 9$, $3\times 3$ have $2/3$ of its columns null,
$$
     \{E\}_3 = \ba{c|cc}
                    1 & 0 & 0 \\
                    0 & 0 & 0 \\
                    0 & 0 & 0 \ea ,\ \
       \{E\}_9 = \ba{ccc|cccccc}
                    1 & 0 & 0 & 0 & 0 & 0 & 0 & 0 & 0\\
                    0 & 0 & 0 & 0 & 0 & 0 & 0 & 0 & 0\\
                    0 & 0 & 0 & 0 & 0 & 0 & 0 & 0 & 0\\
                    0 & 1 & 0 & 0 & 0 & 0 & 0 & 0 & 0\\
                    0 & 0 & 0 & 0 & 0 & 0 & 0 & 0 & 0\\
                    0 & 0 & 0 & 0 & 0 & 0 & 0 & 0 & 0\\
                    0 & 0 & 1 & 0 & 0 & 0 & 0 & 0 & 0\\
                    0 & 0 & 0 & 0 & 0 & 0 & 0 & 0 & 0\\
                    0 & 0 & 0 & 0 & 0 & 0 & 0 & 0 & 0\ea .
$$
These two observations let us obtain an effective representation of $\{\z\}_9$:
$$
  \{\z\}_9 = \{ L(\hat\a) \}_9 \{ E \}_9 \{ L(\hat\a^{(1)}) \}_9 \{ \v \}_9
           = \{ L(\hat\a) \}_9 \{ E \}_{9,3} \{ L(\hat\a^{(1)}) \}_3 \{ \v \}_3.
$$
By using such formula, when $v_0=1$, $v_1=v_2=0$, the vector $\{\z\}_9$ can be computed by performing a $9\times 9$ l.t.T. matrix vector product (if $A$ is $n\times n$ with $n=3^s$ the operations required would have been one product $9\times 9$ l.t.T. matrix $\cdot$ vector $+\ldots+$ one product $3^s\times 3^s$ l.t.T. matrix $\cdot$ vector), that is, about the same amount of operations required by the {\it Gaussian elimination} implemented in the first part of the algorithm.

In conclusion, if $cj3^j$ is an upper bound for the cost of the product $3^j\times 3^j$ l.t.T. matrix $\cdot$ vector, then the overall cost of the shown algorithm is
$\tilde c\sum_{j=2}^s j3^j = O(s3^s) = O(n\log_3 n)$, in case the dimension of the l.t.T. system is $n=3^s$.

Finally observe that a formula more explicit than \eqref{hat_a_3} can be given for the entries of a vector $\hat \a$ such that $L(\a)\hat\a=E\a^{(1)}$. It is reported here below:
{\footnotesize
$$
   \bi{l}
  L(\a)\hat \a \\
  \ \ = \ba{ccccccccccccc}
     1   &        &     &     &     &     &     &     &     &     &     &     &   \\
     a_1 & 1      &     &     &     &     &     &     &     &     &     &     &   \\
     a_2 & a_1    & 1   &     &     &     &     &     &     &     &     &     &   \\
     a_3 & a_2    & a_1 & 1   &     &     &     &     &     &     &     &     &   \\
     a_4 & a_3    & a_2 & a_1 & 1   &     &     &     &     &     &     &     &   \\
     a_5 & a_4    & a_3 & a_2 & a_1 & 1   &     &     &     &     &     &     &   \\
     a_6 & a_5    & a_4 & a_3 & a_2 & a_1 & 1   &     &     &     &     &     &   \\
     a_7 & a_6    & a_5 & a_4 & a_3 & a_2 & a_1 & 1   &     &     &     &     &   \\
     a_8 & a_7    & a_6 & a_5 & a_4 & a_3 & a_2 & a_1 & 1   &     &     &     &   \\
     a_9 & a_8    & a_7 & a_6 & a_5 & a_4 & a_3 & a_2 & a_1 & 1   &     &     &   \\
  a_{10} & a_9    & a_8 & a_7 & a_6 & a_5 & a_4 & a_3 & a_2 & a_1 & 1   &     &   \\
  a_{11} & a_{10} & a_9 & a_8 & a_7 & a_6 & a_5 & a_4 & a_3 & a_2 & a_1 & 1   &   \\
    \cdot &\cdot &\cdot &\cdot &\cdot &\cdot &\cdot &\cdot &\cdot &\cdot &\cdot &\cdot &\cdot \ea
  \ba{c}
     1 \\
     -a_1 \\
     -a_2+a_1^2  \\
     2a_3-a_1a_2  \\
     -a_4-a_1a_3+a_2^2  \\
     -a_5+2a_1a_4-a_2a_3  \\
     2a_6-a_1a_5-a_2a_4+a_3^2  \\
     -a_7-a_1a_6+2a_2a_5-a_3a_4  \\
     -a_8+2a_1a_7-a_2a_6-a_3a_5+a_4^2  \\
     2a_9-a_1a_8-a_2a_7+2a_3a_6-a_4a_5  \\
     -a_{10}-a_1a_9+2a_2a_8-a_3a_7-a_4a_6+a_5^2  \\
     -a_{11}+2a_1a_{10}-a_2a_9-a_3a_8+2a_4a_7-a_5a_6  \\
     \cdot  \ea   \\
  \ \ =   \ba{ccccccccccccc}
     1 &
     0 &
     0  &
     a_1^{(1)} &
     0  &
     0  &
     a_2^{(1)} &
     0  &
     0  &
     a_3^{(1)} &
     0  &
     0  &
     \cdot  \ea^T = E\a^{(1)},  \ei
$$ }
\begin{equation}\label{hat_a_3sc}
   \hat a_i = - \sum_{r=0}^{ \lfloor\f{i-1}{2}\rfloor } a_r a_{i-r} + \de_{i=0\mod 2} a_{\f{i}{2}}^2
       + 3\left\{ \bi{ll}
                     \sum_{s\geq\f{3-i}{6}}^0  a_{\f{i-3}{2}+3s}a_{\f{i+3}{2}-3s}  &  i \ \hbox{odd} \\
                     \sum_{s\geq\f{6-i}{6}}^0  a_{\f{i-6}{2}+3s}a_{\f{i+6}{2}-3s}  &  i \ \hbox{even} \ei \right. ,\ \ i=0,1,2,3,4,5,\ldots .
\end{equation}
Such formula for $\hat a_i$ was found by us looking for the simplest vector $\hat\a$ such that $(L(\a)\hat \a)_i=0$, $i=2,3,5,6,8,9,\ldots$, and it was found before knowing the existence of Theorem \ref{teo:polinomio} and of its easy consequence \eqref{hat_a_3}. As anyone can realize, \eqref{hat_a_3sc} is just the scalar version of
formula \eqref{hat_a_3}.

May be an explicit formula could be given also for the entries of the corresponding vector $\a^{(1)}$. For instance we have:
$$
\bi{l}
a_1^{(1)} = 3a_3-3a_1a_2+a_1^3,\ \
a_2^{(1)} = 3a_6-3a_1a_5-3a_2a_4+3a_3^2-3a_1a_2a_3+3a_1^2a_4+a_2^3,\\
a_3^{(1)} = 3a_9-3a_1a_8-3a_2a_7+6a_3a_6-3a_1a_2a_6-3a_1a_3a_5-3a_2a_3a_4+3a_1^2a_7
+3a_1a_4^2-3a_4a_5+3a_5a_2^2+a_3^3. \ei
$$
The reader could try to obtain the expression for the generic $a_i^{(1)}$ in terms of the $a_j$, i.e. the scalar version of the vector identity $E\a^{(1)}=L(\a)\hat\a=L(\a)L(\c)\d$, $c_k=a_kt^k$, $d_k=a_kt^{2k}$, $t=\exp{\ui 2\pi/3}$.
\vskip0.2cm

{\bf A concluding remark}

We conclude with a remark on the history of the results enclosed in this work. Once the l.t.T. even and odd systems \eqref{pari}, \eqref{dispari} were obtained, we tried to exploit them to retrieve by linear algebra arguments the sparse system, solved by Bernoulli numbers, we observed in the paper \cite{Ramanujan} of Ramanujan (see \eqref{ramanujan-system}). In order to do that, first of all it was necessary to nullify the second, third, fifth, sixth, eighth, ninth, and so on, diagonals of our even and odd systems. So, we naturally conceived the l.t.T. linear systems solvers here presented, and, in particular, the one nullifying at each step $2/3$ of the remaining non null diagonals. Note that our original aim, i.e. write an explicit formula for the vectors $\w^{eR},\w^{oR}\in\CC^{\NN}$ such that $L(\a^e)\w^{eR}=L(\a^o)\w^{oR}=\a^R$, with $\a^R,\a^e,\a^o$ defined in \eqref{typeR}, \eqref{typeEven}, \eqref{typeOdd}, has not been reached in this work. We leave to the reader the interesting exercise to find the vectors $\w^{eR}$ and $\w^{oR}$.

\section*{Appendix A. The l.t.T. matrix-vector product}\label{appendix1}

The product of a $n\times n$ lower triangular Toeplitz matrix times a vector can be computed with  much less than the $n(n+1)/2$ multiplications and $(n-1)n/2$ additions required by the obvious algorithm. The two alternative algorithms here described use the strong relation existing between Toeplitz matrices and the circulant and $(-1)$-circulant \cite{DFTcirc} matrix algebras in order to perform the operation l.t.T. matrix $\cdot$ vector via a small number of \emph{discrete Fourier transforms}, and thus in $O(n\log n)$ arithmetic operations.

\subsubsection*{Preliminaries}
Let $\Pi_{\pm 1}$ be the $n\times n$ matrix $\Pi_{\pm 1} = Z^T \pm \e_n\e_1^T$, where $Z$ is the $n\times n$ lower-shift matrix. Then
\begin{equation}\label{circ}
      \Pi_1 = F D_{1\om^{n-1}} F^* ,\ \ \
      \Pi_{-1} = (D_{1\rho^{n-1}}F) \rho D_{1\om^{n-1}} (D_{1\rho^{n-1}}F)^*
\end{equation}
where $F$ is the following (symmetric) unitary Fourier matrix
$$
F = \f{1}{\sqrt{n}} W,\ \
    W=(\om^{(i-1)(j-1)})_{i,j=1}^n,\ \ \om\ \hbox{such that}\
       \om^n=1,\,\om^i\neq 1,\,0<i<n ,
$$
$D_{1\om^{n-1}}=\diag(1,\om,\ldots,\om^{n-1})$, $\rho$ is such that $\rho^n=-1$, $\rho^i\neq -1,\,0<i<n$, and $D_{1\rho^{n-1}}=\diag(1,\rho,\ldots,\rho^{n-1})$.

From \eqref{circ} it follows that for the circulant and $(-1)$-circulant matrices whose first row is  $\a^T=[a_1\,a_2\cdots a_n]$, that is for the matrices
$C(\a):=\sum_{k=1}^n a_k\Pi_1^{k-1}$ and $C_{-1}(\a):=\sum_{k=1}^n a_k\Pi_{-1}^{k-1}$,
the following representations hold
$$
C(\a) = F d(F^T\a)d(F^T\e_1)^{-1} F^*,\ \ \
C_{-1}(\a) = F_- d(F_-^T\a)d(F_-^T\e_1)^{-1} F_-^*,\ \ F_- = D_{1\rho^{n-1}}F,
$$
where $d(\z)$ denotes the diagonal matrix whose diagonal elements are the entries of the vector $\z$.

Given $\z\in\CCn$, the matrix-vector product $F\z$  is called discrete Fourier transform (DFT) of $\z$. Note that the Fourier matrix satisfies the equalities
$F^2=J\Pi_1$ and $F^*=J\Pi_1 F$, where $J$ is the counter-identity, i.e. the permutation matrix obtained by reversing the columns of the identity matrix. So, the inverse discrete Fourier transform of a vector $\z$, $F^*\z$, is simply a permutation of the DFT of $\z$.
The DFT of $\z$ can be performed through a method, known as FFT, whose computational cost is $O(n\log_b n)$, when $n$ is a power of a number $b$ (see Appendix B). It follows that the same order of arithmetic operations is enough to compute the matrix-vector products $C(\a)\z$ e $C_{-1}(\a)\z$, for any $\a\in\CCn$.
\vskip0.1cm

We are now ready to illustrate two procedures for the computation of the product of a Toeplitz matrix $T=(t_{i-j})_{i,j=1}^n$ times a vector. Obviously such procedures can be applied to our case, where  $t_{k}=0$, $k<0$. We stress the fact that more efficient methods for the computation of l.t.T. matrix-vector products may exist and they would be welcome, being such products the basic operations required by the algorithms presented throughout this work. In fact,
in the previous sections we have seen that the solution of a triangular Toeplitz linear system of $n$ equations, with $n$ power of $2$ $(3)$, can be reduced to the computation of $O(\log_2 n)$ ($O(\log_3 n)$) matrix-vector products, where the matrix involved is Toeplitz triangular and its dimension varies, reducing by a factor $1/2$ ($2/3$) each time. Thus it would be suitable to have a method which performs such products in the most efficient way.

\subsubsection*{Procedure I ($T$ embedded into a circulant)}

Consider a generic Toeplitz $4\times 4$ matrix $T$ and a $4\times 1$ vector $\v$. Then $T$ can be seen as the upper left submatrix of a $8\times 8$ circulant matrix $C$, and the following representation holds for the vector $T\v$:{\small
$$
     T\v = \ba{cccc}
            t_0 & t_{-1} & t_{-2} & t_{-3} \\
            t_1 & t_0    & t_{-1} & t_{-2} \\
            t_2 & t_1    & t_0    & t_{-1 }\\
            t_3 & t_2    & t_1    & t_0    \ea
           \ba{c} v_0 \\ v_1 \\ v_2 \\ v_3 \ea
         = \Big\{ \ba{cccccccc}
               t_0 & t_{-1} & t_{-2} & t_{-3} & 0 & t_3 & t_2 & t_1 \\
               t_1 & t_0 & t_{-1} & t_{-2} & t_{-3} & 0 & t_3 & t_2 \\
               t_2 & t_1 & t_0 & t_{-1} & t_{-2} & t_{-3} & 0 & t_3 \\
               t_3 & t_2 & t_1 & t_0 & t_{-1} & t_{-2} & t_{-3} & 0 \\
               0   & t_3 & t_2 & t_1 & t_0 & t_{-1} & t_{-2} & t_{-3} \\
               t_{-3} & 0& t_3 & t_2 & t_1 & t_0 & t_{-1} & t_{-2} \\
               t_{-2} & t_{-3} & 0 & t_3 & t_2 & t_1 & t_0 & t_{-1} \\
               t_{-1} & t_{-2} & t_{-3} & 0 & t_3 & t_2 & t_1 & t_0 \ea
              \ba{c} v_0 \\ v_1 \\ v_2 \\ v_3 \\ 0 \\ 0 \\ 0 \\ 0 \ea \Big\}_4
           = \big\{ C \ba{c} \v \\ \vn \ea \big\}_4
$$}
where the symbol $\{\z\}_4$ denotes the $4\times 1$ vector whose entries are the first four components of $\z$.

If $T$ is $n\times n$ and $\v$ is $n\times 1$, then the observation still holds, and can be generalized:
$$
   T\v = \big\{ C \ba{c} \v \\ \vn_{(b-1)n} \ea \big\}_n,\ \ \
   C = C(\a)
= \R{bn} F_{bn} d( F_{bn}\a) F_{bn}^H,\ \
   \a=
\ba{c} t_0 \\ t_{-1} \\ \cdot \\ t_{-n+1} \\ \vn_{(b-2)n+1} \\ t_{n-1} \\ \cdot \\ t_1 \ea.
$$
If $n$ is a power of $b$ ($b=2,3,\ldots$), from such formula one immediately deduces a procedure of cost $O(n\log_b n)$ for the computation of the product of a $n\times n$ Toeplitz matrix times a  vector (see Appendix B).

\subsubsection*{Procedure II ($T$ written as the sum of a circulant and a (-1)-circulant)}

Set $\a=[a_1\,\cdots\,a_n]^T$ and $\a^\prime=[a_1^\prime\,\cdots\,a_n^\prime]^T$ where $a_i=\um(t_{-i+1}+t_{n-i+1})$, $a_i^\prime=\um(t_{-i+1}-t_{n-i+1})$, $i=1,\ldots,n$ ($t_n=0$).
Then, the following representation holds for our Toeplitz matrix  $T=(t_{i-j})_{i,j=1}^n$:
$$
    T = C(\a) + C_{-1}(\a^\prime)
      = F d(F^T\a)d(F^T\e_1)^{-1} F^* + F_- d(F_-^T\a^\prime)d(F_-^T\e_1)^{-1} F_-^* .
$$
Again, if $n$ is a power of $b$ ($b=2,3,\ldots$), from this formula one immediately deduces a procedure of cost $O(n\log_b n)$ for the computation of the product of a  $n\times n$ Toeplitz matrix times a  vector (see Appendix B).


\section*{Appendix B. The FFT algorithm}\label{appendix2}

\begin{proposition}[(FFT)] Let $n$ be a power of $b$ ($b=2,3,\ldots$). Given $\z\in\CCn$, the DFT of $\z$ can be computed in at most $O(n\log_b n)$ arithmetic operations.
\end{proposition}
\begin{proof}
Let $n$ be such that $b|n$. Since $\om^{(i-1)(k-1)}$ is the $(i,k)$ element of $W$ and $z_k$ is the $k$-th element of $\z\in\CCn$, we have
$$
	\bi{l}
	(W\z)_i  =  \sum_{k=1}^n \om^{(i-1)(k-1)} z_k \\
	          =
                 \sum_{j=1}^{n/b} \om^{(i-1)(bj-b)} z_{bj-b+1}
	           + \sum_{j=1}^{n/b} \om^{(i-1)(bj-b+1)} z_{bj-b+2}
               + \ldots
               + \sum_{j=1}^{n/b} \om^{(i-1)(bj-b+b-1)} z_{bj-b+b}  \\
	     =
                \sum_{j=1}^{n/b} (\om^b)^{(i-1)(j-1)} z_{bj-b+1}
	         + \sum_{j=1}^{n/b} \om^{(i-1)(b(j-1)+1)} z_{bj-b+2}
               + \ldots
               + \sum_{j=1}^{n/b} \om^{(i-1)(b(j-1)+b-1)} z_{bj-b+b}\\
	       =
               \sum_{j=1}^{n/b} (\om^b)^{(i-1)(j-1)} z_{bj-b+1}
	       + \om^{i-1}\sum_{j=1}^{n/b} (\om^b)^{(i-1)(j-1)} z_{bj-b+2}
                  + \ldots
     + \om^{(i-1)(b-1)} \sum_{j=1}^{n/b} (\om^b)^{(i-1)(j-1)} z_{bj-b+b}.
	\ei
$$
Note that $\om$ is actually a function of $n$, in fact $\om$ is such that $\om^n=1$ and $\om^i\neq 1$, $0<i<n$. So, a better notation for $\om$ is $\om_n$. Then $\om^b=\om_n^b$ is such that $(\om_n^b)^{n/b}=1$ and $(\om_n^b)^{i}\neq 1$, $0<i<n/b$; in other words $\om_n^b=\om_{n/b}$ (namely  $\om_n^b$ is the $n/b$-th principal root of the unity). Thus we have the identity
	\begin{equation}\label{FFTscalar}
   \bi{rcl}
	   (W_n\z)_i & = & \sum_{j=1}^{n/b} \om_{n/b}^{(i-1)(j-1)} z_{bj-b+1}
	         + \om_n^{i-1}\sum_{j=1}^{n/b} \om_{n/b}^{(i-1)(j-1)} z_{bj-b+2} \\
                & &  + \ldots
           + \om_n^{(i-1)(b-1)} \sum_{j=1}^{n/b} (\om_{n/b})^{(i-1)(j-1)} z_{bj-b+b},\ \
	               i=1,2,\ldots,n. \ei
	\end{equation}
It follows that, for $i=1,\ldots,\f{n}{b}$,
	$$
	   (W_n\z)_i = ( W_{n/b} \ba{c} z_1 \\ z_{b+1} \\ \cdot \\ z_{n-b+1} \ea )_i
	          + \om_n^{i-1} ( W_{n/b} \ba{c} z_2 \\ z_{b+2} \\ \cdot \\ z_{n-b+2} \ea )_i
              + \ldots
      + \om_n^{(i-1)(b-1)} ( W_{n/b}\ba{c} z_b \\ z_{2b} \\ \cdot \\ z_{n} \ea )_i.
	$$
Moreover, letting $i=\f{n}{b}+k$, $k=1,\ldots,\f{n}{b}$, in \eqref{FFTscalar}, we obtain
	$$
	\bi{rcl}
	(W_n\z)_{\f{n}{b}+k} & = &
	\sum_{j=1}^{n/b} \om_{n/b}^{\f{n}{b}(j-1)}
	                   \om_{n/b}^{(k-1)(j-1)} z_{bj-b+1}
	         + \om_n^{\f{n}{b}}\om_n^{k-1}
	\sum_{j=1}^{n/b} \om_{n/b}^{\f{n}{b}(j-1)}
	                    \om_{n/b}^{(k-1)(j-1)} z_{bj-b+2}  \\
    & &          + \ldots
              + \om_n^{\f{n}{b}(b-1)}\om_n^{(k-1)(b-1)}\sum_{j=1}^{n/b}
               \om_{n/b}^{\f{n}{b}(j-1)}
	                    \om_{n/b}^{(k-1)(j-1)} z_{bj}   \\
                    \\
	        & = &
		\sum_{j=1}^{n/b}
	                   \om_{n/b}^{(k-1)(j-1)} z_{bj-b+1}
	         + \om_n^{\f{n}{b}}\om_n^{k-1}
	\sum_{j=1}^{n/b}
	                    \om_{n/b}^{(k-1)(j-1)} z_{bj-b+2}  \\
      & &        + \ldots
              + \om_n^{\f{n}{b}(b-1)}\om_n^{(k-1)(b-1)}\sum_{j=1}^{n/b}
	                    \om_{n/b}^{(k-1)(j-1)} z_{bj}\\
                  \\
	       & = &
             ( W_{n/b} \ba{c} z_1 \\ z_{b+1} \\ \cdot \\ z_{n-b+1} \ea )_k
	         + \om_n^{\f{n}{b}}\om_n^{k-1}
                 ( W_{n/b} \ba{c} z_2 \\ z_{b+2} \\ \cdot \\ z_{n-b+2} \ea )_k  \\
            &  &  + \ldots
              + \om_n^{\f{n}{b}(b-1)}\om_n^{(k-1)(b-1)}
                 ( W_{n/b} \ba{c} z_b \\ z_{2b} \\ \cdot \\ z_n  \ea  )_k,\ \
	     k=1,\ldots,\f{n}{b},   \ei
	$$
where $\om_n^{\f{n}{b}}= \om_b$. Proceeding in this way, one obtains formulas for
$(W_n\z)_{r\f{n}{b}+k}$, $r=0,1,\ldots,b-1$, $k=1,\ldots,\f{n}{b}$.
Such scalar equalities can be written in a more compact form:
	\begin{equation}\label{FFTvector}
   W_n \z = \ba{cccc}
              I & D           & \cdot & D^{b-1} \\
              I & \om_b D     & \cdot & (\om_bD)^{b-1} \\
          \cdot & \cdot       & \cdot & \cdot \\
              I & \om_b^{b-1}D & \cdot & (\om_b^{b-1}D)^{b-1} \ea
             \ba{cccc}
               W_{n/b} &         &       &  \\
                       & W_{n/b} &       &  \\
                       &         & \cdot &  \\
                       &         &       & W_{n/b} \ea Q,
	\end{equation}
where
$$
   D = \ba{cccc}
         1 &       &       &                   \\
           & \om_n &       &                   \\
           &       & \cdot &                   \\
           &       &       & \om_n^{\f{n}{b}-1}\ea ,
$$
and $Q$ is the permutation matrix such that
$$
   Q\z = [z_1\ z_{b+1}\ \cdot\ z_{n-b+1}\
          z_2\ z_{b+2}\ \cdot\ z_{n-b+2}\
          \cdots\
          z_b\ z_{2b}\ \cdot\ z_{n}]^T.
$$
By the previous formula \eqref{FFTvector}, it is clear that a $W_n$ transform is computable by performing $b$ $W_{n/b}$ transforms. So, if $c_n$ denotes the complexity of the matrix-vector product $W_n\z$, then $c_n \leq b c_{n/b} + 2(b-1) n$,
which implies $c_n=O(n\log_b n)$, if $n$ is a power of $b$.
\end{proof}

\section*{Appendix C. The detailed l.t.T. linear system solver algorithm}\label{appendix3}

Preliminary definitions:
$$
    \a = \ba{c}
           1  \\
           a_1 \\
           a_2 \\
           \cdot \ea \in \CC^{\NN},\ \
    L(\a) = \ba{cccc}
              1   &     & & \\
              a_1 & 1   & & \\
              a_2 & a_1 & 1 & \\
              \cdot & \cdot & \cdot & \cdot \ea ,
$$
$$
     E = \ba{ccccc}
            1   &  0  &  0  & \cdot & \cdot \\
            \vn & \vn & \vn & \cdot & \cdot \\
            0   &  1  &  0  & \cdot & \cdot \\
            \vn & \vn & \vn & \cdot & \cdot \\
            0   &  0  &  1  & 0     & \cdot \\
          \cdot & \cdot & \cdot & \cdot & \cdot \ea,\ \vn=\vn_{b-1}, \ \
     E^s =    \ba{ccccc}
            1   &  0  &  0  & \cdot & \cdot \\
            \vn & \vn & \vn & \cdot & \cdot \\
            0   &  1  &  0  & \cdot & \cdot \\
            \vn & \vn & \vn & \cdot & \cdot \\
            0   &  0  &  1  & 0     & \cdot \\
          \cdot & \cdot & \cdot & \cdot & \cdot \ea,\ \vn=\vn_{b^s-1}, \
$$
$$
       \u = \ba{c}
           1  \\
           u_1 \\
           u_2 \\
           \cdot \ea,\ \
     E\u = \ba{c}
             1 \\
             \vn \\
             u_1 \\
             \vn \\
             u_2 \\
             \cdot \ea,\ \vn=\vn_{b-1},\ \
     L(E\u) = \ba{cccccc}
              1   &     & & & & \\
              \vn & I   & & & & \\
              u_1 & \vn^T & 1 & & & \\
              \vn & u_1 I & \vn & I & & \\
              u_2 & \vn^T & u_1 & \vn^T & 1 & \\
              \cdot & \cdot & \cdot & \cdot & \cdot & \cdot \ea,\         \vn=\vn_{b-1}.
$$
\vskip0.2cm

A generalization of Lemmas \ref{lemma2} and \ref{lemma2_npower3}:
\vskip0.2cm

{\bf Lemma}: If $\u,\v\in \CC^{\NN}$ and $u_0=v_0=1$, then
$L(E\u) E\v = E L(\u)\v$, $L(E^s\u)E^s\v = E^s L(\u)\v$, $\fa\,s\in\NN$.
\vskip0.2cm

Now, by using the above Lemma and Lemma \ref{lemma1}, we are ready to present an
algorithm for the computation of $\x$ such that $A\x=\e_1$ where $A$ is a $n\times n$ l.t.T. matrix with $n=b^k$ and $[A]_{11}=1$. The overall cost of the algorithm is
$O(n \log_b n)$.

First of all observe that the $n\times n$ matrix $A$ can be thought as the upper-left submatrix of a semi-infinite l.t.T. matrix $L(\a)$, whose first column is
$[1\, a_1\, a_2\,\cdot\, a_{b^k-1}\, a_{b^k}\,\cdot\,]^T$.
$$
     L(\a)= \ba{cc}
             A & O \\
             \vdots & \ddots \ea
          = \ba{ccccc}
             1 & & & & \\
             a_1 & 1 & & & \\
             \cdot & \cdot & \cdot & & \\
             a_{b^{k}-1} & \cdot& a_1 & 1 & \\
             \cdot & \cdot & \cdot & \cdot & \cdot \ea,\
     A\in\CC^{b^k\times b^k},\ \
     \a^{(0)}:=\a
$$
FIRST PART:
\\[0.2cm]
\noindent
Step 1: Find $\hat\a^{(0)}$, $\a^{(1)}$ such that
$$
    L(\a^{(0)})\hat\a^{(0)}
        = E\a^{(1)} = \ba{c} 1 \\ \vn \\ a^{(1)}_1 \\ \cdot \ea,\ \vn=\vn_{b-1}, \
        \hbox{ so that }\
    L(\hat\a^{(0)})L(\a^{(0)}) = L(E\a^{(1)}).
$$
Step 2: Find $\hat\a^{(1)}$, $\a^{(2)}$ such that
$$
    L(\a^{(1)})\hat\a^{(1)}
        = E\a^{(2)} =
        \ba{c} 1 \\ \vn \\ a^{(2)}_1 \\ \cdot \ea,\ \vn=\vn_{b-1}, \ \ \hbox{ so that}
$$
$$
     L(E\a^{(1)})E\hat\a^{(1)}
        = E^2\a^{(2)} = \ba{c} 1 \\ \vn \\ a^{(2)}_1 \\ \cdot \ea,\ \vn=\vn_{b^2-1},\ \
     L(E\hat\a^{(1)}) \underline{ L(E\a^{(1)}) }
        = L(E^2\a^{(2)}).
$$
Step 3: Find $\hat\a^{(2)}$, $\a^{(3)}$ such that
$$
    L(\a^{(2)})\hat\a^{(2)}
        = E\a^{(3)} = \ba{c} 1 \\ \vn \\ a^{(3)}_1 \\ \cdot \ea,\ \vn=\vn_{b-1}, \ \
        \hbox{ so that}
$$
$$
     L(E^2\a^{(2)})E^2\hat\a^{(2)}
        = E^3\a^{(3)} = \ba{c} 1 \\ \vn \\ a^{(3)}_1 \\ \cdot \ea,\ \vn=\vn_{b^3-1},\ \
     L(E^2\hat\a^{(2)}) \underline{ L(E^2\a^{(2)}) }
        = L(E^3\a^{(3)}).
$$
$\ldots$

\noindent
Step $k$: Find $\hat\a^{(k-1)}$, $\a^{(k)}$ such that
$$
    L(\a^{(k-1)})\hat\a^{(k-1)}
        = E\a^{(k)} = \ba{c} 1 \\ \vn \\ a^{(k)}_1\\ \cdot \ea,\ \vn=\vn_{b-1}, \ \
        \hbox{ so that}
$$
$$
     L(E^{k-1}\a^{(k-1)}) E^{k-1}\hat\a^{(k-1)}
    = E^{k}\a^{(k)} = \ba{c} 1 \\ \vn \\ a^{(k)}_1 \\ \cdot \ea,\ \vn=\vn_{b^{k}-1},\ \
     L(E^{k-1}\hat\a^{(k-1)}) \underline{ L(E^{k-1}\a^{(k-1)}) }
        = L(E^{k}\a^{(k)}) .
$$
Summarizing, we have the identity
$$
     L(E^{k-1}\hat\a^{(k-1)}) L(E^{k-2}\hat\a^{(k-2)}) \,\cdots\, L(E\hat\a^{(1)}) L(\hat\a^{(0)}) L(\a^{(0)}) = L(E^k\a^{(k)})
$$
where the upper left $b^k\times b^k$ submatrices of $L(\a^{(0)})$ and of $L(E^k\a^{(k)})$ are, respectively, the initial l.t.T. matrix $A$ and the identity matrix.
\vskip0.2cm

\noindent
SECOND PART:
\vskip0.1cm

\noindent
Note that, for any $\c\in\CC^{\NN}$,
$$
    L(\a^{(0)})\z=\c \ \hbox{ iff } \
    L(E^k\a^{(k)}) \z =
    L(\hat\a^{(0)}) L(E\hat\a^{(1)}) \,\cdots\, L(E^{k-2}\hat\a^{(k-2)}) L(E^{k-1}\hat\a^{(k-1)}) \c  .
$$
Moreover, if
$$
\c=E^{k-1}\v = \ba{c}
         v_0 \\
         \vn \\
         v_1 \\
         \vn \\
         v_{2} \\
         \vn \\
         \cdot \ea, \ \vn=\vn_{b^{k-1}-1},\ \ \v\in\CC^{\NN},
$$
then :
$$
    L(\a^{(0)})\z=\c \ \hbox{ iff } \
    \ba{cc}
        {\large I}_{b^k}   & O \\
                           &   \\
       \ba{cc}
         a^{(k)}_1 & \\
                   & \cdot \ea  & \ddots  \ea \z
    = L(E^k\a^{(k)}) \z =
L(\hat\a^{(0)}) E L(\hat\a^{(1)}) E\,\cdots\, E L(\hat\a^{(k-2)}) E L(\hat\a^{(k-1)}) \v .
$$
In other words, any vector $\big\{\z\big\}_n$, $n=b^k$, such that
$$
    A\{\z\}_n = \{L(\a)\}_n \{\z\}_n
     =  \ba{c}
         v_0 \\
         \vn \\
         v_1 \\
         \vn \\
         \cdot \\
         v_{b-1} \\
         \vn \ea, \ \vn=\vn_{b^{k-1}-1}
$$
(for example, if $v_0=1,\,v_i=0$ $i\geq 1$, the vector we are looking for, $A^{-1}\e_1$),
can be represented as follows
$$
   \bi{rcl}
    \{ \z \}_n & = & \{  L(\hat\a^{(0)})  \}_n  \{  E  \}_n
              \{  L(\hat\a^{(1)})  \}_n  \{  E  \}_n \,\cdots\,
              \{  L(\hat\a^{(k-2)})  \}_n  \{  E  \}_n
              \{  L(\hat\a^{(k-1)})  \}_n   \{\v\}_n  \\
            & = & \{  L(\hat\a^{(0)})  \}_n  \{  E  \}_{n,\f{n}{b}}
              \{  L(\hat\a^{(1)})  \}_{\f{n}{b}}  \{  E  \}_{\f{n}{b},\f{n}{b^2}} \,\cdots\,
              \{  L(\hat\a^{(k-2)})  \}_{\f{n}{b^{k-2}}}  \{  E  \}_{\f{n}{b^{k-2}},\f{n}{b^{k-1}}}
              \{  L(\hat\a^{(k-1)})  \}_{\f{n}{b^{k-1}}}  \{\v\}_b   \ei
$$
The latter formula allows us to compute $\{ \z \}_n$ efficiently.

Let us resume and count the operations required.
In the following, $n$ is equal to $b^k$ and $\vn$ denotes $\vn_{b-1}$.
\vskip0.2cm

First, for $j=0,\ldots,k-1$ we have to compute, by performing
{\large $\varphi_{\f{n}{b^j}}$} arithmetic operations, the vectors
$I^1_{\f{n}{b^j}} \hat \a^{(j)}$
and
$I^1_{ \f{n}{b^{j+1}} }\a^{(j+1)}$, i.e. scalars $\hat a^{(j)}_i$
and $a^{(j+1)}_i$ such that
$$
\bi{l}
\underbrace{
   \ba{ccccc}
     1         &           &    &       &  \\
     a_1^{(j)} & 1         &    &       &  \\
     a_2^{(j)} & a_1^{(j)} & 1  &       &  \\
     \cdot     &  \cdot    & \cdot & \cdot &  \\
    a_{\f{n}{b^j}-1}^{(j)} & \cdot & a_2^{(j)} & a_1^{(j)} & 1 \ea }
   \ba{c}
      1 \\
      \hat a_1^{(j)} \\
      \hat a_2^{(j)} \\
      \cdot       \\
      \hat a_{\f{n}{b^j}-1}^{(j)} \ea
  = \ba{c}
      1 \\
      \vn \\
      a_1^{(j+1)} \\
      \vn \\
      \cdot       \\
      a_{\f{n}{b^{j+1}}-1}^{(j+1)} \\
      \vn \ea,\ \ j=0,\ldots,k-1 \\
       \qquad \qquad\quad  \displaystyle{ \f{n}{b^j} \ \times\ \f{n}{b^j} } \ei
$$
(note that there is no $a^{(k)}_i$ to be computed).
\vskip0.1cm

\noindent
{\it Remark}. The
$\f{n}{b^j} \ \times\ \f{n}{b^j}$ l.t.T. by vector products, $j=0,\ldots,k-2$, that one has to perform in order to compute the $I^1_{ \f{n}{b^{j+1}} }\a^{(j+1)}$, can be in fact replaced with a number $b$ of $\f{n}{b^{j+1}}\times\f{n}{b^{j+1}}$ l.t.T. by vector products, $j=0,\ldots,k-2$.

\vskip0.2cm\noindent
Second, we have to compute the $b\times b$ l.t.T. by vector product
$\{L(\hat \a^{(k-1)})\}_{\f{n}{b^{k-1}}} [ v_0 \ \cdot\cdot\cdot \ v_{b-1}]^T$,
and $\f{n}{b^j}\times\f{n}{b^j}$ l.t.T. by vector products of type
$\{ L(\hat \a^{(j)}) \}_{\f{n}{b^j}}
[ 1\ \vn^T\ \bullet\ \vn^T \cdot\cdot\cdot \ \bullet\ \vn^T]^T$,
$j=k-2,\ldots,1,0$.
\vskip0.2cm

If we assume the cost of a $b^j\times b^j$ l.t.T. by vector product and $\varphi_{b^j}$ both bounded by $c b^{j}j$ where $c$ is a constant (we know that this is true at least for $b=2,3$), then the total cost of the above operations is smaller than $O(b^k k)=O(n\log_b n)$. In particular, if $v_0=1$, $v_i=0$, $i>0$, by such operations we obtain the first column of $A^{-1}$, or, in other words, a l.t.T. linear system solver of complexity $O(n\log_b n)$.
\vskip0.2cm

{\it Note}

Some of the contents of this work have been the subject of a communication held at the 2012--edition of the annual italian meeting ``Due Giorni di Algebra Lineare Numerica'' (Genova, 16--17 Febbraio 2012; speaker: Carmine Di Fiore). See www.dima.unige.it/$\sim$dibenede/2gg/home.html
\vskip0.2cm

{\it Acknowledgements}

Thanks to professor Wolf Gross who taught to the first author Bernoulli numbers and their beautiful properties, to professor Dario Bini who pointed us the problem of polynomial arithmetic related with the algorithms presented, and to the Rome-Moscow school 2012 which gave the authors the opportunity to teach, and then to write, in the present form, their studies on the numerical linear algebra of Bernoulli numbers.


\end{document}